\documentclass[a4paper,12pt]{amsart}
\usepackage{amsmath,amssymb,amsthm,float}

\usepackage[maxbibnames=5]{biblatex}
\usepackage{enumitem}

\usepackage[colorlinks=true, citecolor=blue, linkcolor=blue, urlcolor=blue]{hyperref}
\usepackage{algorithm}
\usepackage{pbox}
\usepackage[indLines=false,noEnd=false]{algpseudocodex}

\newcommand{\Gate}[2]{\BeginBox\State \pbox{\textwidth}{#1}\ \ \pbox{\textwidth}{#2}\EndBox}
\newcommand{\Register}[2]{{\em #1:} #2}

\usepackage{todonotes}

\let\mod\relax
\DeclareMathOperator{\mod}{mod}

\usepackage{fancyhdr}
\usepackage{epsfig,amsfonts,latexsym}
\usepackage{color}

\usepackage{changepage}
\usepackage{caption}
\usepackage{subcaption}

 \usepackage{url} 
 \usepackage{framed}
\usepackage{amsmath, color}
\usepackage{amssymb}
\usepackage{amssymb}
\usepackage{latexsym}
 \usepackage{graphicx, epstopdf, verbatim}
 \graphicspath{ {./images/} }
\usepackage{psfrag}

\newcommand{\figcommented}[1]{}

\def \beq {\begin {eqnarray}}
\def \eeq {\end {eqnarray}}
\def \ba {\begin {eqnarray*}}
\def \ea {\end  {eqnarray*}}

\newcommand{\be}{\begin{equation}}
\newcommand{\ee}{\end{equation}}

\numberwithin{equation}{section}

\newtheorem{theorem}{Theorem}

\newtheorem{lemma}{Lemma}[section]

\newtheorem{proposition}[lemma]{Proposition}

\theoremstyle{definition}

\newtheorem{problem}{Problem}

\theoremstyle{remark}

\makeatother

\newcommand{\ignore}[1]{}

\newcommand{\barr}{\begin{array}}
\newcommand{\earr}{\end{array}}

\def\bfo{\begin{eqnarray*}}
\def\efo{\end{eqnarray*}}
\def\ba{\begin{eqnarray*}}
\def\ea{\end{eqnarray*}}
\def\beq{\begin{eqnarray}}
\def\eeq{\end{eqnarray}}

\renewcommand{\H}{{\mathbb H}}

\def\ket{\rangle}

\def\picture #1 by #2 (#3){
   \vsquare to #2{
     \hrule width #1 height 0pt depth 0pt
     \hfill
     \special{picture #3}}}

\def\scaledpicture #1 by #2 (#3 scaled #4){{
   \dimen0=#1 \dimen1=#2
   \divide\dimen0 by 1000 \multiply\dimen0 by #4
   \divide\dimen1 by 1000 \multiply\dimen1 by #4
   \picture \dimen0 by \dimen1 (#3 scaled #4)}}

\def \C {{\mathbb {C}}}

\def \be {e}

\def \H2s {H^{s+1}_0(\partial M\times [0,T/2])}

\def \pa0 {\partial _0}

\def \mbeq {\begin {eqnarray}}
\def \meeq {\end {eqnarray}}

\def \be {{ {e}}}

\title[Quantum computing algorithms for inverse problems] {Quantum computing algorithms for inverse problems on graphs and an NP-complete inverse problem}

\author[Ilmavirta, Lassas,  Lu,  Oksanen, Ylinen]{Joonas Ilmavirta, Matti Lassas, Jinpeng Lu, Lauri Oksanen, Lauri Ylinen}

\address{}\curraddr{}

\begin{document}
	\date{\today}

\maketitle

\tableofcontents 

{\bf Abstract.} {\it We consider an inverse problem for a finite graph $(X,E)$
where we are given a subset of vertices $B\subset X$ and the distances $d_{(X,E)}(b_1,b_2)$  of all vertices $b_1,b_2\in B$. The distance of points $x_1,x_2\in X$ is defined as the minimal number of edges needed to connect two vertices, so all edges have length 1. The inverse  problem is a discrete version of the boundary rigidity problem in Riemannian geometry or the inverse travel time problem in geophysics. We will show that this problem has unique solution under certain conditions and develop  quantum computing methods to solve it. We prove the following uniqueness result:  when $(X,E)$ is a tree and $B$ is the set of leaves of the tree, the graph $(X,E)$ can be uniquely determined in the class of  all graphs having a fixed number of vertices.  We present a quantum computing algorithm which produces a graph $(X,E)$, or one of those, which has a given number of vertices and the required distances between vertices in $B$. To this end we develop an  algorithm that takes in a qubit representation of a graph and combine it with Grover's search algorithm. The algorithm can be implemented using only $O(|X|^2)$ qubits, the same order as the number of elements in the adjacency matrix of $(X,E)$. It also has a quadratic improvement in computational cost compared to standard classical algorithms.  Finally, we consider applications in theory of computation, and show that a slight modification of the above inverse problem is NP-complete: all NP-problems can be reduced to a discrete inverse problem we consider.}

\section{Introduction}

We consider inverse travel time problems for graphs that are analogous
to the boundary rigidity problems studied in Riemannian geometry or the inverse travel time problems studied in seismic imaging. 
The boundary rigidity problem asks if a Riemannian manifold with boundary can be determined by the knowledge of the distances between boundary points.
The problem was conjectured \cite{MR636880} to be uniquely solvable for simple manifolds (i.e., with convex boundary and no conjugate points), and proved in two dimensions in \cite{PU05}. In higher dimensions, boundary rigidity is known for simple conformal metrics \cite{MR78,Croke91}, 
generic simple manifolds including all analytic ones \cite{SU05}, metrics close to Euclidean or hyperbolic \cite{Gromov,CDS,LSU,BI,BI13}, and for manifolds foliated by strictly convex hypersurfaces \cite{SUV16,SUV21}.
See also the surveys \cite{Ivanov10,Survey_Uhlmann,SUVZ,PSU}.

The boundary rigidity problem originally arose from geophysics as the inverse travel time problem or the inverse kinematic problem \cite{H1907,W1910}, where the goal is to reconstruct the interior structure of the body from the travel time measurements of waves on its boundary.
In seismic imaging, the speeds of seismic waves define a metric structure, and the distance between points is the travel time of the waves between the points.

We are interested in a discrete version of the inverse travel time problem. This type of problem has applications in the determination of phylogenetic tree of species in biology \cite{SS,Fel03,DBP} (a branching diagram or a tree showing the evolutionary relationships among various species based upon similarities in their physical or genetic characteristics), in chemistry 
\cite{Wiener,EJS76,BCM,ZT} (notably the Wiener index), and in tomography of networks such as internet tomography \cite{CGGS,CCLNY,CK13}.
See also the surveys \cite{AH,Survey_index}.
The problem is also related to inverse problems for random walks \cite{WSBT,MPL,BILL,BILL2} and a discrete version of a Calder{\'o}n type problem \cite{GR}.
In particular, the random walk measurements used in \cite{BILL2} contain more data than the distance data considered in the present paper.
Related to these applications, we prove a uniqueness result showing that
a tree with $n$ vertices can be uniquely reconstructed from the leaf-to-leaf distances in the class of all connected graphs with $n$ vertices.
Furthermore, in this paper we demonstrate how quantum computing algorithms can be used to solve discrete inverse problems. We will also show that a certain generalized inverse travel time problem for graphs is an NP-complete problem.

\subsection{Formulation of the inverse problems for graphs} \label{subsection-formulation}

Let $(X,E)$ be a finite undirected simple graph with $B\subset X$  being the set of ``observation nodes'', where $X$ is the set of vertices and $E$ is the set of edges.
We denote this by $(X,E,B)$.
Recall that a graph is said to be simple if there is at most one edge between any pair of vertices and there is no edge that connects a vertex to itself.
For undirected simple graphs,
edges are two-element subsets of $X$.
Two vertices $x,y\in X$ are called adjacent if $\{x,y\}\in E$, i.e., there is an edge between $x$ and $y$.
Two graphs are said to be \emph{isomorphic} if there is an edge-preserving bijection, i.e., \emph{graph isomorphism}, between the sets of vertices of the graphs.
The \emph{graph distance} function on a graph $(X,E)$, denoted by $d_{(X,E)}$ or $d_E$ for short,
is defined to be the minimal number of edges in paths that connect two vertices.
In particular, any pair of adjacent vertices has distance $1$. If there is no path connecting vertices $x$ and $y$, then we set $d_E(x,y) = \infty$.
Then a graph can be considered as a metric space equipped with the graph distance.
Two graphs are isomorphic if and only if they are isometric as metric spaces.

\medskip 

\begin{problem}[Inverse travel time problem for graphs]
    \label{prob:inverse-travel-time}
    Given an integer $n\in\mathbb{N}=\{0,1,2,\dots\}$, a finite set $B$ and a function $d_0:B\times B\to \mathbb N$, find a graph $(X,E)$ such that $|X|=n$ (that is, the graph has $n$~vertices), $B\subset X$ and that the graph distance $d_E|_{B\times B}:B\times B\to \mathbb N$ is the function $d_0$.
\end{problem}

In particular, we are interested in whether a solution to the problem, if it exists, is unique up to a graph isomorphism.
Note that we consider this problem within the class of graphs with fixed total number of vertices.
This is because without fixing the number of vertices the solution to the problem is never unique, as one can add isolated vertices to the graph without changing the graph distance at all. 
If all solutions  $(X,E)$  to the above inverse problem are isometric, we say
that the graph  $(X,E,B)$ is \emph{$B$-observation rigid} in the class of graphs with $n$ vertices.

\smallskip
We also consider a generalized version of the above problem. In the generalization we may specify that certain edges \emph{must} appear in the solution, and that certain edges \emph{must not} appear in the solution; we also allow the boundary distance function to be only partially defined on the boundary nodes:

\begin{problem}[Restricted inverse travel time problem for graphs]
    \label{prob:restricted-inverse-travel-time}
    Given a finite set $X$, sets $V\subset X\times X$ and $E_0\subset E_1\subset \{\{x_1,x_2\} : {x_1,x_2}\in X\}$, and a function $d_{0,V}:V\to \mathbb N$, find a set of edges $E$ such that $E_0\subset E\subset E_1$ and that the restriction $d_E|_V:V\to \mathbb N$ of the graph distance function of $(X,E)$ coincides with $d_{0,V}.$
\end{problem}

\section{Rigidity results for trees}

In general, one cannot expect a solution to the inverse travel time problem for graphs, if it exists, to be unique even up to a graph isomorphism. A simple counterexample would be a triangle and a path of three vertices, with $B$ being a pair of adjacent vertices.
Strong assumptions on $B$ or $d_0$ are likely necessary to make the solution unique.

As a first step in solving this uniqueness, or rigidity, problem, we consider the case where $d_0$ is induced by a tree metric on the set $B$ of vertices of degree $1$ (that is,
$B$ and $d_0:B\times B\to \mathbb N$ are such that there exists a tree $(X_0,E_0)$ containing $n$ vertices, $B\subset X_0$ is the set of leaves of $X_0$ and $d_0(b_1,b_2)=
d_{(X_0,E_0)}(b_1,b_2)$ for all $b_1,b_2\in B$).
Recall that a tree is a connected graph without cycles, and a vertex of degree $1$ in a tree is called a \emph{leaf}.
The degree of a vertex $v$ in a graph $(X,E)$ is the number of edges incident to $v$, denoted by $\hbox{deg}_{(X,E)} (v)$ or $\hbox{deg}_E (v)$ for short.
In this case, we will prove that the inverse travel time problem with such $B,d_0$ is uniquely solvable in the class of graphs with fixed total number of vertices.
We emphasize that we do not \emph{a priori} assume that the unknown graph $(X,E)$ is a tree.
We formulate this rigidity result as follows.

For $m\leq n$, consider the following class of finite connected undirected simple graphs $(X,E,B)$ with a subset $B\subset X$:
\begin{equation}\label{graph-class}
\begin{split}
\mathcal{G}(n,m)
&=
\big\{(X,E,B): |X|=n, \, |B|=m,\, \text{and}
\\&\quad
\textrm{deg}_E (b)=1 \textrm{ for all }b\in B \big\}.
\end{split}
\end{equation}
Here $|\cdot|$ denotes the cardinality.
Note that also vertices in $X\setminus B$ may have degree $1$.

\begin{theorem}[Boundary rigidity for trees]\label{tree-rigidity}
Let $T=(\bar X, \bar E, \bar B) \in \mathcal{G}(n,m)$ be a tree, where $\bar B$ is the set of all leaves of the tree. For any $G=(X,E,B)\in \mathcal{G}(n,m)$, if there exists an identification of $B$ and $\bar B$ such that $d|_{B\times B}=\bar d|_{\bar B\times \bar B}$ where $d,\bar d$ denote the graph distance functions on $G,T$, then $G$ is isomorphic to $T$.
\end{theorem}

The proof of Theorem \ref{tree-rigidity} is given in Section~\ref{subsection-rigidity}.
We remark that in the literature, the set of leaves is often chosen to act as the "boundary" for a tree for various purposes, and Theorem \ref{tree-rigidity} can be formulated in short that \emph{trees are boundary rigid in $\mathcal{G}(n,m)$ for fixed $n,m$}. The key point is that we do not \emph{a priori} assume $G$ is a tree: we allow $G$ to be in the class of all connected simple graphs. If one \emph{a priori} knows $G$ is also a tree, the result reduces to (the uniqueness part of) the classical Tree-metric Theorem (e.g. \cite{Buneman71,Buneman74,HY,S69,Z65}).

Trees can be seen as the combinatorial analogue to simple manifolds, in the sense that the shortest path between any pair of vertices in trees is unique. 
Note that the meaning of "simple" is different for graphs and manifolds.
In fact, trees have even stronger property that any pair of vertices in a tree can be connected by a unique path 
that does not repeat vertices.

\section{Quantum algorithm for the inverse travel time problem}

The basic unit of information in quantum computing is the \emph{qubit}---a two-level quantum mechanical system. The two levels of the qubit are mathematically represented by two orthonormal vectors denoted by ${|0\rangle,|1\rangle}\in\C^2$; a general state of the qubit is a unit vector in $\C^2$. A \emph{quantum register} is a system comprising multiple qubits; the state of a quantum register consisting of $N$ qubits is described by a unit vector in the $N$-fold tensor product $(\C^2)^{\otimes N}$. To every classical bit string $x\in\{0,1\}^N$ there corresponds the quantum state
\begin{equation}
    \label{eq:computational-state}
    |x\rangle 
    = |x_1x_2\cdots x_N\rangle
    =|x_1\rangle\otimes|x_2\rangle\otimes\dots\otimes|x_N\rangle \in (\C^2)^{\otimes N};
\end{equation}
states of the form~\eqref{eq:computational-state} form the \emph{computational basis} of $(\C^2)^{\otimes N}$.

Quantum computation consists of a sequence of elementary operations applied to a quantum register, accompanied by a measurement in the end. The elementary operations are represented by \emph{quantum gates}, which are unitary operators on $(\C^2)^{\otimes N}$ that act on a small number of qubits. The result of the measurement is a classical bit string; if the state of the quantum register immediately before the measurement is $|\psi\rangle$, then a bit string  $x\in\{0,1\}^N$ is observed in the measurement with probability $| \langle x | \psi\rangle |^2$. 

A fixed sequence of quantum gates (optionally followed by a measurement) that is used as a part of an algorithm describing quantum computation is called a \emph{quantum subroutine}. A sequence of quantum gates and a measurement together constitute a \emph{quantum circuit}. For basic concepts related to quantum algorithms, we refer to the paper \cite{Nannicini}. For a treatise, see~\cite{MR1796805}.

\subsection{Grover's algorithm}
\label{sec:grover}
\newcommand{\kkket}[1]{|#1\rangle} 
\newcommand{\bbbra}[1]{\langle #1|}
\newcommand{\bbbraket}[2]{\langle#1|#2\rangle}

Grover's algorithm~\cite{MR1427516} is an algorithm that solves a quantum version of following unstructured search problem: Let $N$ be a positive integer and $\{T, F\}$ be a partition of the set $\{0,1\}^N$.
This means that $\{0,1\}^N = T \cup F$ and $T\cap F =\emptyset$, not excluding the possibility that $T=\emptyset$ or $F=\emptyset$.
The elements of $T$ are called solutions to the search problem. For each $x\in\{0,1\}^N$ it is possible to query a function for whether $x\in T$ or $x\notin T$. The function is viewed as a black box and it is called an \emph{oracle}. The problem is to find a solution $x\in T$ (if there are any) by querying the oracle as few times as possible. In the worst case finding a solution requires querying all possibilities, so the classical query complexity (the minimum number of queries to find a solution in the worst case) of the problem is $2^N$.

In the quantum version of the problem the oracle is replaced by a quantum oracle which may be queried by arguments that are in a superposition. More precisely, a \emph{quantum oracle} (corresponding to a set $T$) is a quantum subroutine that maps a computational basis state of the form $\kkket{x}\otimes\kkket{r}\otimes\kkket{0}_k$ (where $x\in\{0,1\}^N$, $r\in\{0,1\}$ and $\kkket{0}_k=\kkket{0}\otimes\dots\otimes\kkket{0}\in(\C^2)^{\otimes k}$ are $k$ extra qubits) as
\begin{equation}
  \label{eq:oracle-operator}
  \kkket{x}\otimes\kkket{r}\otimes\kkket{0}_k \mapsto 
  \kkket{x}\otimes\kkket{r + f(x)\mod 2}\otimes \kkket{0}_k,
\end{equation}
where $f(x)=1$ if $x\in T$, $f(x)=0$ otherwise.  In other words, the quantum oracle uses the qubit $\kkket{r}$ to mark whether $x\in T$ or $x\notin T$. Here the extra qubits $\kkket{0}_k$ are called \emph{ancillae} and they are used as a working memory. Grover's algorithm is a method that finds a solution (if such exists) in this quantum setting with high probability with only $O(2^{N/2})$ queries to the quantum oracle. 

In Section~\ref{sec:quantum-oracle} below we will construct a quantum oracle for the inverse travel time problem for graphs (Problem~\ref{prob:inverse-travel-time}). After that, in Section~\ref{sec:grover-for-inverse-problem} we will explain how Grover's algorithm together with the oracle constructed in Section~\ref{sec:quantum-oracle} can be used to find a solution to Problem~\ref{prob:inverse-travel-time}. The standard version of Grover's algorithm assumes that the number of solutions, $|T|$, is known; this is not true in the case of Problem~\ref{prob:inverse-travel-time}. For the convenience of the reader, in Appendix~\ref{app: Grover} we recall (with proof) how Grover's algorithm can be applied in the case of an unknown number of solutions.

\subsection{Quantum oracle for the inverse travel time problem} 
\label{sec:quantum-oracle}

We will consider a quantum algorithm where the adjacency matrix of the graph $(X,E)$ is represented by a sequence of qubits. 
Our algorithm computes if the distances between the vertices in the subset $B\subset X$ satisfy the required conditions given by the data $d_0$ in Problem~\ref{prob:inverse-travel-time}.

To transform the inverse travel time problem for graphs to a computational problem, we enumerate the elements of $X$ as $\{1,2,\dots,n\}$ and the elements of $B$
as $\{1,2,\dots,m\}$.
Let us consider a binary-valued function $e:\{1,2,\dots,n\}^2\to\{0,1\}$
that satisfies $e(j,j)=0$ and $e(j,k)=e(k,j)$. Such a function specifies a graph
$(X,E)$, defined by 
$$
\{x_j,x_k\}\in E \iff e(j,k)=1 .
$$
Our aim is to find values of $e(j,k)$ for which the corresponding graphs 
$(X,E)$ with a subset $B\subset X$ satisfy $d_E|_{B\times B}=d_0$.

Without loss of generality, we assume that the data $d_0$ and $n$ given in Problem~\ref{prob:inverse-travel-time} satisfy $d_0(j,j)=0$ for all $j\in B$, $d_0(j,k)=d_0(k,j)$ and $d_0(j,k)\ge 1$ for all ${j,k}\in B$, $j\neq k$, and also  $\max_{{j,k}\in B} d_0(j,k) \le n-1$. This is possible without loss of generality, as these conditions can easily be checked by a fast classical algorithm, and they clearly are necessary for the existence of a solution. In our implementation of the algorithm, we assume that $|B|=m \ge 4$.

Before introducing the quantum algorithm, we consider two classical algorithms, Algorithm~\ref{alg_paths} and Algorithm~\ref{alg_classicaltest} below, which test if a given graph is a solution to the inverse
travel time problem for graphs, and then consider the analogous quantum algorithms.
In the classical algorithms we consider variables  having values in $\{0,1\}$ (i.e., classical bits), or equivalently, TRUE/FALSE.

Algorithm~\ref{alg_paths} takes in as an input the vector $\vec e=(e(i,j))_{1\le i<j\le n}$,
where $e(i,j)\in \{0,1\}$ (in fact, in the algorithm we consider the extended variable $e(i,j)$, defined for all $i\not = j$, that satisfy $e(j,i)=e(i,j)$),
and computes distances from a specific node $o\in B$.
The distances
are computed using a graph $(X,E)$
that corresponds to the adjacency matrix $\vec e$.

Algorithm~\ref{alg_classicaltest} tests if the distances in the graph $(X,E)$
corresponding to the adjacency matrix $\vec e$ satisfy
the conditions given in the data $d_0$ for the inverse problem, that is, whether $d_E(i,j)=d_0(i,j)$,
for all $i,j\in B$.

In Algorithm~\ref{alg_paths}, we compute, for a fixed node $o \in B$,
the truth values $p(d,j)$, $d = 1,\dots, |X| - 1$, $j \in X \setminus \{o\}$, such that
\begin{equation}\tag{P}
\label{eq:P}
\begin{split}
  &p(d,j) = 1 \text{ if~and~only~if there is a path of length $\le d$}\\
  &\text{between $o$ and $j$.}  
\end{split}
\end{equation}
\newpage

\begin{algorithm}[H]
\caption{Compute shortest paths from $o \in B$}\label{alg_paths}
\begin{algorithmic}
\Require adjacency matrix $e(j,k)$, $j, k \in X$ and $o\in B$
\Ensure $p(d,j) \in \{0, 1\}$, $d = 1,\dots, n - 1$, $j \in X \setminus \{o\}$, satisfies (P)
\ForAll{$j \in X \setminus \{o\}$}
    \If{$e(o,j)=1$}
        \State $p(1,j) \gets 1$
    \Else
        \State $p(1,j) \gets 0$
    \EndIf
\EndFor
\For{$d = 2,\dots,n - 1$}
    \ForAll{$j \in X \setminus \{o\}$}
        \ForAll{$k \in X \setminus \{o,j\}$}
            \If{$p(d-1, k) = 1$ and $e(k, j) = 1$}
                \State $a(k) \gets 1$
            \Else
                \State $a(k) \gets 0$
            \EndIf
        \EndFor 
        \If{$p(d-1,j) = 1$ or $a(k) = 1$ for some $k \in X \setminus \{o,j\}$}
            \State $p(d,j) \gets 1$
        \Else
            \State $p(d,j) \gets 0$
        \EndIf
    \EndFor 
\EndFor 
\end{algorithmic}
\end{algorithm}
\newpage

\begin{algorithm}[H]
\caption{Test if a graph $(X,E)$  has the correct $B$-distances}\label{alg_classicaltest} 
\begin{algorithmic}
\Require adjacency matrix $e(j,k)$, $j, k \in X$, and  numbers $d_0(j,k)$, $j,k \in  B $
\Ensure $d_E(j,k)=d_0(j,k)$ for $j,k \in  B $
\ForAll{$o \in B$}
       \\   
$\hbox{\ }$\quad {\bf Call} Algorithm~\ref{alg_paths} {\bf with} \\
$\hbox{\ }$\quad \quad\quad Input:   $e(j,k)$, $j, k \in X$ and $o\in B$\\
$\hbox{\ }$\quad \quad\quad Output:  $p(d,j)$, $d=1,\dots,n-1$, $j\in X\setminus \{o\}$   
        \ForAll{$j \in B \setminus \{o\}$}
              \State $d \gets d_0(o,j)$
            \If{$d=1$ and $p(1,j)=1$}
                \State $s(j) \gets 1$
            \ElsIf{$d\ge 2$ and $p(d-1, j) = 0$ and $p(d, j)  = 1$}
                \State $s(j) \gets 1$
            \Else
                \State $s(j) \gets 0$
            \EndIf
            \EndFor 
                 \State $t(o) \gets  s(1) \,\&\, s(2) \,\&\, \cdots \,\&\, s(o-1) \,\&\, s(o+1) \,\&\, \cdots \,\&\, s(m)$
        \EndFor 
           \State $r \gets  t(1)\, \& \,t(2)\, \& \,\cdots \, \& \, t(m)$

\Output r
\end{algorithmic}
\end{algorithm}

Let us next introduce some tools that will allow writing quantum versions of our algorithms.
We will denote quantum gates and quantum subroutines by boxes containing the name of the quantum gate or subroutine and the qubits on which it operates. The qubits are listed in three categories:
``Controls'', ``Targets'', and ``Ancillae''. If the quantum register is in a computational basis state before the quantum gate/subroutine is applied, then the values of the control qubits are not changed, whereas the values of the target qubits may change. The ancilla qubits are assumed to be in the state $|0\rangle$ at the beginning of the 
subroutine, during the execution of the subroutine their values may change, but their values are returned to the original value $|0\rangle$ by uncomputation at the end of the subroutine. The controls, targets, and ancillae correspond to input, output, and working memory of a classical algorithm, respectively.

Below, we denote the classical $\{0,1\}$-valued bit by lower case letter, such as $c_1$ and qubits by capital letter, such as $C_1$.

A basic example of a gate is the double controlled NOT gate, CCNOT, also called the Toffoli gate.
We write this gate as 
\begin{algorithmic}
\Gate{CCNOT}{
\Register{Controls}{$C_1$ and $C_2$}\\
\Register{Target}{$T$}
}
\end{algorithmic}
The CCNOT gate has two control qubits $C_1$ and $C_2$, one target qubit $T$ and no ancillae. When viewed as a Boolean function the CCNOT gate is given by
    \ba
(c_1, c_2, t) \mapsto (c_1, c_2, (t + c_1 \cdot c_2) \mod 2),\quad c_1, c_2, t\in \{0,1\}.
    \ea
As a quantum gate, CCNOT maps a computational basis state $|C_1 C_2 T\rangle$ to the computational basis state $|C_1\rangle\otimes|C_2\rangle\otimes|(T + C_1 \cdot C_2) \mod 2\rangle$.

The NOT gate acts on a single qubit as the unitary operator
    \begin{eqnarray*}
X |0\rangle = |1\rangle, \quad 
X |1\rangle = |0\rangle.
    \end{eqnarray*}
For the NOT gate, there is no reason to label the qubit and we write simply

\begin{centering}
\begin{algorithmic}
\Gate{NOT}{$T$
}
\end{algorithmic}
\end{centering}
for the NOT gate operating on the qubit $T$.

The Hadamard gate $H$ also acts on a single qubit, its action is specified by
\begin{equation*}
    H|0\rangle = \frac{1}{\sqrt{2}} (|0\rangle + |1\rangle),\quad
    H|1\rangle = \frac{1}{\sqrt{2}} (|0\rangle - |1\rangle).
\end{equation*}

The operation performed by an arbitrary quantum gate $G$ is invertible, we denote the gate that performs the inverse operation by $REVERSE-G$, and we use similar notation for quantum subroutines. Even though CCNOT gate is its own inverse, we sometimes
write $REVERSE-CCNOT=CCNOT$ to emphasize that we do uncomputing steps.

As simple examples, we describe in Appendix~\ref{appendix:iterated-and-or} how  the OR gate with two qubits can be implemented in terms of the AND and NOT gates, 
and how to implement the quantum iterated AND operator and OR operator that correspond to the classical operators
\ba
& &\hbox{AND${}_{(m)}$}(c_1,c_2,\dots,c_m)=c_1\,\hbox{AND}\,c_2\,\hbox{AND}\,\dotsb \,\hbox{AND}\, c_m,\\& &
\hbox{OR${}_{(m)}$}(c_1,c_2,\dots,c_m)=c_1\,\hbox{OR}\,c_2\,\hbox{OR}\,\dotsb \,\hbox{OR}\, c_m.
\ea

\subsubsection{Pseudocodes of quantum algorithms}

Next, we consider quantum versions of Algorithms~\ref{alg_paths} and~\ref{alg_classicaltest} above.
Let $(X,E)$ be a graph, and $B\subset X$ be a subset where distances are observed.
As before, we identify the set $X$ of $n$ vertices by $\{1,2,\dots,n\}$, and the set $B\subset X$ of $m$ vertices by $\{1,2,\dots,m\}$.
We consider qubits $E(j,k)$, $j, k \in X$, that have the state $ |1\rangle$ if $j,k$ are adjacent (i.e., $\{j,k\}\in E$) in the graph $(X,E)$, and have the state $ |0\rangle$ if $j,k$ are not adjacent (i.e., $\{j,k\}\notin E$). We always assume that $j\not =k$,
and when the set $ X$ is identified with $\{1,2,\dots,n\}$, these qubits are indexed as
 $E(j,k)$, where $1\le j<k\le n$. When we refer in the algorithms below to a qubit 
  $E(j,k)$ where $j>k$, we mean the qubit   $E(k,j)$.

We start with procedures that our main algorithm ORACLE uses.
Later, Grover's algorithm calls the ORACLE as a subroutine.

The subroutine ORACLE operates on the following qubits:

$E(j,k)$ for $1\le j< k\le n$ that encodes an adjacency matrix in $X\times X$;

\smallskip

$R$ is the target qubit.
\medskip

Moreover, ORACLE uses the following ancilla qubits having initial state $|0\rangle$:

\smallskip
$P(d,j)$ for $d = 1,\dots, n - 1$, $j=1,\dots, n$;
\smallskip

$A(k)$ for $k=1,\dots, n$;
\smallskip

$F(k)$ for $k =1,\dots, n$;

\smallskip

$T(o)$ for $o =1,\dots, m$.

\medskip

Next, we give pseudocodes of the subroutines for the ORACLE.
We give these in very precise form to show how ancilla qubits 
are used so that the total number of qubits used is $O(n^2)$.

The following subroutine UPDATE takes in the current 
value of the distance parameter $d$ and the number
$o\in B$ of the starting vertex from which we consider  the
distances to the other vertices $x\in X\setminus \{o\}$.
The subroutine UPDATE operates on the qubits $E(j,k)$ corresponding to the adjacency matrix and the qubits $P(d-1,j)$ and $P(d,j)$, where $j \in X \setminus \{o\}$. Assuming that the qubits $P(d-1,j)$ satisfy a quantum version of Condition~\eqref{eq:P} for the distances computed from the node $o$, UPDATE changes the values of the qubits $P(d,j)$, $j \in X \setminus \{o\}$, so that also they satisfy the same condition.
As the operator OR${}_{(n)}$ can be realized by $O(n)$ CCNOT gates, the subroutine UPDATE can be implemented using  $O(n^2)$ CCNOT gates.
\newpage

\begin{algorithm}[H]
\caption{UPDATE subroutine}\label{gate_update} 
\begin{algorithmic}
\Require 
\Register{Parameters}{$n,\ d\in \{2,3,\dots,n-1\},o\in B$}\\
{\bf Operates on qubits}\\
\Register{Control 1}{$E(j,k)$, $j, k \in X \setminus \{o\}$}
\Comment{adjacency matrix}\\
\Register{Control 2}{$P(d-1,j)$, $j \in X \setminus \{o\}$}
\Comment{Is $d(o,j)\leq d-1$?}\\
\Register{Targets}{$P(d,j)$, $j \in X \setminus \{o\}$}
\Comment{``Output''}\\
\Register{Ancillae}{$A(k)$, $F(k)$, $k \in X \setminus \{o\}$\\}

    \ForAll{$j \in X \setminus \{o\}$}
        \ForAll{$k \in X \setminus \{o,j\}$}

            \Gate{CCNOT}{
\Register{Controls}{$P(d-1, k)$ and $E(k, j)$}\\
\Register{Target}{$A(k)$}
}
\EndFor

               \LComment{
We view $X \setminus \{o,j\}$ as being synonymous to $\{1,\dots,n -2\}$}     

\Gate{OR${}_{(n -1)}$}{
\Register{Controls}{ $A(k)$, for $k\in X \setminus \{o,j\}$, and \\
$\hbox{\ }$\hspace{20mm}$P(d-1,j)$}\\
\Register{Target}{$P(d,j)$}\\
\Register{Ancillae}{$F(k)$, $k = 1,\dots,n-3$}
}
\ForAll{$k \in X \setminus \{o,j\}$}
            \Comment{Uncompute ancillae}
\Gate{REVERSE-CCNOT}{
\Register{Controls}{$P(d-1, k)$ and $E(k, j)$}\\
\Register{Target}{$A(k)$}
}
        \EndFor 
    \EndFor 

\end{algorithmic}
\end{algorithm}
\newpage

The following subroutine PATHS takes in the number
$o\in B$ of the starting vertex from which we consider  the
distances to the other vertices $x\in X\setminus \{o\}$.
The subroutine PATHS iterates the subroutine UPDATE and it
operates
on the qubits $E(j,k)$ corresponding to the adjacency matrix and the qubits $P(d,j)$, where $j \in X \setminus \{o\}$
and $d=1,2,\dots,n-1$. The subroutine PATHS 
changes the values of all qubits $P(d,j)$
so that they satisfy the quantum version of Condition~\eqref{eq:P} for the distances computed from the node $o$.
As UPDATE consists of  $O(n^2)$ CCNOT gates, the subroutine PATHS
can be implemented using  $O(n^3)$ CCNOT gates.

\begin{algorithm}[H]
\caption{PATHS subroutine}\label{gate_paths} 
\begin{algorithmic}
\Require \Register{Parameters}{$n,o$}\\
{\bf Operates on qubits}\\
\Register{Controls}{$E(j,k)$ for $j, k \in X \setminus \{o\}$}
\Comment{adjacency matrix}\\
\Register{Targets}{{$P(d,j)$ for $d = 1,\dots, n - 1$, $j \in X \setminus \{o\}$}}\\
\Register{Ancillae}{$A(k)$, $F(k)$ for $k \in X \setminus \{o\}$\\}
\LComment{UPDATE cannot be used to initialize $P(1,j)$, therefore this case is treated separately:}

 \ForAll{$j \in X \setminus \{o\}$}
    \Gate{CNOT}{
        \Register{Control}{$E(j,o)$}\\
        \Register{Target}{$P(1,j)$}
    }
\EndFor
\LComment{UPDATE computes $P(d,j)$ for $d\ge 2$:}
\For{$d = 2,\dots,n- 1$}      
        \Gate{UPDATE}{
\Register{Parameters}{$n$, $d$, $o$}\\
\Register{Controls 1}{$E(j,k)$ for $k \in X \setminus \{o\}$, $j \in X \setminus \{o\}$}\\
\Register{Controls 2}{$P(d-1,j)$ for $j \in X \setminus \{o\}$}\\
\Register{Targets}{$P(d,j)$ for $j \in X \setminus \{o\}$}\\
\Register{Ancillae}{$A(k)$, $F(k)$ for $k \in X \setminus \{o\}$}
}
\EndFor 
\end{algorithmic}
\end{algorithm}

In the following subroutine TEST, we consider  the
distances from all vertices $o\in B$ to the other vertices $b \in B\setminus \{o\}$.
The subroutine TEST iterates the subroutine PATHS, and it
operates
on the qubits $E(j,k)$ corresponding to the adjacency matrix and the qubits $T(o)$, where $o$ runs over all
vertices in $B$. The qubits  $T(o)$ tell if the distances
from the vertex $o$ to the other vertices $b \in B\setminus \{o\}$
have the correct values. We point out that the
uncomputing of the qubits $P(d,j)$ inside the subroutine TEST is the reason why
the ORACLE uses only $O(n^2)$ qubits. 
As PATHS consists of  $O(n^3)$ CCNOT gates, the subroutine TEST
can be implemented using  $O(m n^3)$ CCNOT gates.

\begin{algorithm}[H]
\caption{TEST subroutine}\label{gate_test}
\begin{algorithmic}
\Require
    \Register{Parameters}{$d_0(j,k)$ for $j, k \in B$}
    
    {\bf Operates on qubits}\\
    \Register{Controls}{$E(j,k)$ for $j, k \in X$}\Comment{adjacency matrix}\\
    \Register{Targets}{$T(o)$ for $o \in B$}\Comment{``Output''}\\
    \Register{Ancillae}{$A(k)$, $F(k)$ for $k \in X$}\\
\Register{Ancillae}{$P(d,j)$ for $d = 1,\dots, n - 1$}, $j \in X$\\
\For{$o\in B$}
    \Gate{PATHS}{
        \Register{Parameters}{$n$, $o$}\\
        \Register{Controls}{$E(j,k)$ for $j, k \in X\setminus\{o\}}$\\
        \Register{Targets}{$P(d,j)$ for $d = 1,\dots, n - 1$, $j \in X \setminus \{o\}$}\\
        \Register{Ancillae}{$A(k)$, $F(k)$ for $k \in X \setminus \{o\}$}
    }
    \LComment{Since PATHS uncomputes $A(k)$ and $F(k)$, they can be reused as ancillae below.}
\For{$j\in B\setminus \{o\}$}
\State $d \gets d_0(j,o)$
        \LComment{Determine if $d_E(j,o) = d_0(j,o)$:}
        \If{$d=1$}
            \Gate{CNOT}{
                \Register{Control}{$P(d,j)$}\\
                \Register{Target}{$A(j)$}
            }
        \Else
            \Gate{NOT}{$P(d-1,j)$}
            \Gate{CCNOT}{
                \Register{Controls}{$P(d, j)$ and $P(d-1, j)$}\\
                \Register{Target}{$A(j)$}
            }
        \EndIf
    \EndFor
\\
    \Gate{AND${}_{(m-1)}$}{
\Register{Controls}{$A(j),\ j\in B\setminus \{o\}$}\\
        \Register{Target}{$T(o)$}\\
        \Register{Ancillae}{$F(k)$, $k = 1,\dots,m-3$}
    }
\LComment{Uncompute $A(j)$ and restore $P(d-1,j)$ to its original value:}
    \For{$j\in B\setminus \{o\}$}
        \State $d \gets d_0(j,o)$
        \If{$d=1$}
            \Gate{CNOT}{
                \Register{Control}{$P(d,j)$}\\
                \Register{Target}{$A(j)$}
            }
        \Else
            \Gate{REVERSE-CCNOT}{
                \Register{Controls}{$P(d, j)$ and $P(d-1, j)$}\\
                \Register{Target}{$A(j)$}
            }
            \Gate{NOT}{$P(d-1,j)$}
        \EndIf
    \EndFor  
\LComment{Uncompute $P(d,j)$:}
    \Gate{REVERSE-PATHS}{
        \Register{Parameters}{$n$, $o$}\\
        \Register{Controls}{$E(j,k)$ for $j, k \in X\setminus\{o\}$}\\
        \Register{Targets}{$P(d,j)$, $d = 1,\dots, n - 1$, $j \!\in\! X\! \setminus \!\!\{o\}$}\\
        \Register{Ancillae}{$A(k)$, $F(k)$ for $k \in X \setminus \{o\}$}
    }
\EndFor 
\end{algorithmic}
\end{algorithm}

The following subroutine ORACLE operates on the qubits $E(j,k)$ corresponding to the adjacency matrix and the `output' qubit $R$. The value of $R$ 
is changed according to whether the graph corresponding
to the adjacency matrix $E(j,k)$ is such that  the distances
between all vertices in $B$
have the correct values. We point out that the subroutine ORACLE uses $O(n^2)$ qubits. 
As TEST consists of  $O(mn^3)$ CCNOT gates, ORACLE
can be implemented using  $O(m n^3)$ CCNOT gates.

\begin{algorithm}[H]
\caption{ORACLE subroutine}
\label{alg:quantum-oracle}
\begin{algorithmic}
\Require 
\Register{Parameters}{$n=|X|,$ $m=|B|,$ and $d_0(j,k)$ for $j, k \in B$}

{\bf Operates on qubits}\\
\Register{Controls}{$E(j,k)$ for $j, k \in X$}
\Comment{adjacency matrix}\\
\Register{Target}{$R$}
\Comment{``output''}\\
\Register{Ancillae}{{$P(d,j)$ for $d = 1,\dots, n - 1$, $j \in X$}}\\
\Register{Ancillae}{$A(k)$, $F(k)$ for $k \in X$\\}
\Register{Ancillae}{$T(o)$ for $o \in B$}\\

\Gate{TEST}{
    \Register{Parameters}{$d_0(j,k)$ for $j, k \in B$}\\
   \Register{Controls}{$E(j,k)$ for $j, k \in X$}\\
\Register{Targets} {$T(o)$ for $o \in B$}\\
\Register{Ancillae}{$A(k)$, $F(k)$ for $k \in X$\\}
\Register{Ancillae}{$P(d,j)$ for $d = 1,\dots, n - 1$, $j \in X$}
}      
\vspace{2mm}
        \Gate{AND${}_{(m)}$}{
\Register{Controls}{ $T(o),\ o\in B $}\\
\Register{Target}{$R$}\\
\Register{Ancillae}{$F(k)$, $k = 1,\dots,m-2$}
}
\vspace{2mm}
\Gate{REVERSE-TEST}{
    \Register{Parameters}{$d_0(j,k)$ for $j, k \in B$}\\
   \Register{Controls}{$E(j,k)$ for $j, k \in X$}\\
\Register{Targets} {$T(o)$ for $o \in B$}\\
\Register{Ancillae}{$A(k)$, $F(k)$ for $k \in X$\\}
\Register{Ancillae}{$P(d,j)$ for $d = 1,\dots, n - 1$, $j \in X$}
}
\end{algorithmic}
\end{algorithm}

\subsection{Quantum algorithm for Problem~\ref{prob:inverse-travel-time}}
\label{sec:grover-for-inverse-problem}

In order to apply Grover's algorithm to the inverse travel time problem for graphs we identify the search space $\{0,1\}^N$, $N=n(n-1)/2$, with the set of all graphs $(X,E)$ that have $|X|=n$ vertices. In this identification a vector $\vec{e} = e(j,k)_{1\le j<k\le n}$ specifies an adjacency matrix, which in turn specifies a graph. The set $T$ of solutions to the search problem consists of those vectors $\vec{e}$ for which the corresponding graph $(X,E)$ is a solution to Problem~\ref{prob:inverse-travel-time}. The quantum oracle corresponding to the set $T$ is the ORACLE subroutine (Algorithm~\ref{alg:quantum-oracle}).

Grover's algorithm consist of a repeated application of ORACLE subroutine together with another subroutine called DIFFUSION.
The target qubits for DIFFUSION are the $N$ qubits $E(j,k)$ that describe the adjacency matrix, and on these qubits it acts as the operator
\begin{equation}
    \label{eq:diffusion-subroutine}
    D_N := 2\kkket{\gamma_N}\bbbra{\gamma_N} - \mathrm{id}^{\otimes N},
\end{equation}
where $\mathrm{id}$ is the identity operator and $\kkket{\gamma_N}$ is the uniform superposition
\begin{equation}
    \label{eq:gamma-N}
    \kkket{\gamma_N} := \frac{1}{\sqrt{2^N}} \sum_{x\in\{0,1\}^N} \kkket{x}.
\end{equation}
See Appendix~\ref{app: Grover} for more details. DIFFUSION subroutine can be realized with $O(N)$ gates and $O(N)$ ancillae; see~\cite{MR1907291}.

Both ORACLE and DIFFUSION subroutines expect the ancillae to be initialized to the state $|0\ket$ in the beginning, and in the end they uncompute the ancillae back to the state $|0\ket$. As these subroutines are run sequentially, they may use the same ancillae. Both  subroutines require $O(N)$ ancillae; we will label the set of the ancillae for the subroutines by $S(k)$, $k=1,2,\dots,O(N)$. In addition, the quantum register contains the target qubit $R$ for the ORACLE.

Subroutine GROVER-L below  describes Grover's algorithm with $L$ iterations.
The main part of the algorithm consists of a repeated application of ORACLE and DIFFUSION subroutines. A key question is to decide the number $L$ of applications performed. If the number of solutions of the search problem, $|T|$, is \emph{a priori} known, then it is well known that an optimal value of $L$ can be calculated; see Appendix~\ref{app: Grover}. In our case the number of solutions, i.e., the number of graphs $(X,E)$ for which the graph distance function coincides with the given data $d_0$ on the set $B\subset X$, is \emph{a priori} unknown. In GROVER-L we choose to take $L$ as one of the input parameters to the algorithm; in Algorithm~\ref{alg:Grover-with-unknown-number-of-solutions} below we will show how GROVER-L can be used as a subroutine in an algorithm that solves the search problem with an unknown number of solutions.

The HADAMARD subroutine below consists of $N+1$ Hadamard gates; the subroutine applies a Hadamard gate to each of the $N+1$ target qubits individually.

\newpage

\begin{algorithm}[H]
    \caption{GROVER-L}
    \label{alg:Grover-L-iterations}
    \begin{algorithmic}
    \Require \Register{Parameters}{$L$, $n$, $m$, $d_0(j,k)$ for ${j,k}\in B$}\\
    \State initialize all qubits to $|0\rangle$
    \Gate{NOT}{$R$}
    \Gate{HADAMARD}{
            \Register{Target}{$R$ and $E(j,k)$ for $j, k \in X$}
        }
    \For{$\ell=1,2,\dots,L$}
        \Gate{ORACLE}{
            \Register{Parameters}{$n$, $m$, $d_0(j,k)$ for ${j,k}\in B$}\\
            \Register{Controls}{$E(j,k)$ for $j, k \in X$}\\
            \Register{Target} {$R$}\\
            \Register{Ancillae}{$S(k)$, $k=1,\dots,O(N)$}
        }
        \Gate{DIFFUSION}{
            \Register{Target}{$E(j,k)$ for $j, k \in X$}\\
            \Register{Ancillae}{$S(k)$, $k=1,\dots,O(N)$}
        }
    \EndFor
    \State $\vec{e} \gets$ measurement of $E(j,k)$ for ${j,k}\in X$
    \Comment{$\vec{e} = e(j,k)_{1\le j<k\le n}$}
    \State \Output $\vec{e}$
\end{algorithmic}
\end{algorithm}

Algorithm~\ref{alg:Grover-with-unknown-number-of-solutions} below consists of a repeated application of GROVER-L as a subroutine, where the value of $L$ is varied between the applications. After each application of GROVER-L it is tested if the obtained value $\vec{e}$ is a solution to Problem~\ref{prob:inverse-travel-time}. This test can be computed classically by Algorithm~\ref{alg_classicaltest}. If $\vec{e}$ is a solution, then the algorithm outputs $\vec{e}$ and halts.

The function $L(\theta)$ appearing in the algorithm is defined by
\begin{equation}
\label{eq:Grover-L-theta}
    L(\theta) := 
    \begin{cases}
        \left\lceil\frac{\pi}{4\theta} - \frac{1}{2}\right\rceil,\, \text{ if } 0 < \theta < \frac{\pi}{8},\\
        1,\text{ if } \frac{\pi}{8} \le \theta < \frac{\pi}{4},\\
        0, \text{ if } \frac{\pi}{4} \le \theta < \frac{\pi}{2}.
    \end{cases}
\end{equation}

\newpage

\begin{algorithm}[H]
    \caption{Quantum algorithm for inverse travel time problem for graphs}
    \label{alg:Grover-with-unknown-number-of-solutions}
    \begin{algorithmic}
    \Require \Register{Parameters}{$n$, $m$, $d_0(j,k)$ for ${j,k}\in B$}
    
    \State $N \gets n(n-1)/2$
    \For{$L=0,1,2$}
        \State $\vec{e} \gets$ GROVER-L with parameter $L$
        \If{$\vec{e}$ is a solution}
            \State \Output $\vec{e}$ and halt
        \EndIf
    \EndFor
    \For{$K=2^{N-4},2^{N-5},\dots,2,1$}
        \State find $\theta\in(0,\pi/2)$ such that $\sin\theta = \sqrt{K / 2^N}$ 
        \State $L\gets$ \Call{$L$}{$\theta$}
        \State $\vec{e} \gets$ GROVER-L with parameter $L$
        \If{$\vec{e}$ is a solution}
            \State \Output $\vec{e}$ and halt
        \EndIf
    \EndFor
    \State \Output ``no solution found''
\end{algorithmic}
\end{algorithm}

Properties of Grover's algorithm are discussed in Appendix~\ref{app: Grover}; see Propositions \ref{prop:Grover-evolution-of-state} and \ref{prop:Grover-with-unknown-number-of-solutions}. Combining these results with the ORACLE algorithm, we obtain our main result:

\begin{theorem}
    \label{thm:oracle and Grover-evolution-of-state}
    Let there be given $n=|X|$, $m=|B|$ and the target distance function $d_0:B\times B\to \mathbb N$ for the inverse travel time problem for graphs (Problem~\ref{prob:inverse-travel-time}). Then the following holds:
    \begin{enumerate}
    \item If Algorithm~\ref{alg:Grover-with-unknown-number-of-solutions} outputs a solution $\vec{e}$, then the solution is correct. This means that the graph $(X,E)$, where the set of edges $E$ is specified by $\vec{e}$, is such that the graph distance function on the set $B\subset X$ coincides with the given data $d_0$, that is, $d_{E}(x,y)=d_0(x,y)$ for all $x,y\in B$.
    \end{enumerate}
    
    Let $N=n(n-1)/2$, choose $\delta\in (0,1)$, and let $M>\log_2(\delta^{-1})$ be an integer. If the data $n$, $m$ and $d_0$ are such that Problem~\ref{prob:inverse-travel-time} admits a solution, then also the following hold:
    \begin{enumerate}[resume]
    \item  With probability greater than or equal to $1/2$, Algorithm~\ref{alg:Grover-with-unknown-number-of-solutions} finds a solution $\vec{e}$ with
    \begin{equation*}
        O\bigg(2^{N/2}
        \bigg(\frac 1e (N^{\frac 12}-m)\bigg)^{-(N^{\frac 12}-m)/2}
        \bigg)
    \end{equation*}
    queries to the quantum subroutine ORACLE (Algorithm~\ref{alg:quantum-oracle}) and to the classical Algorithm~\ref{alg_classicaltest}.
    \item If Algorithm~\ref{alg:Grover-with-unknown-number-of-solutions} is applied $M$ times, then at least one solution is found with probability greater than $1-\delta$.
    \end{enumerate}

    The  quantum subroutine ORACLE can be implemented with $O(N)=O(n^2)$ qubits and $O(m n^3)$ gates from the set $\{\text{NOT, CNOT, CCNOT}\}$.
\end{theorem}

\begin{proof}
\begin{enumerate}
    \item Algorithm~\ref{alg:Grover-with-unknown-number-of-solutions} outputs a solution only after having verified it is correct.
    \item We note that even in the case when the inverse problem of finding a graph $(X,E)$ that satisfies $d_E|_{B\times B}=d_0$ is uniquely solvable, that is, the graph is uniquely determined up to an isomorphism, the number of solutions $\vec e$ is $(n-m)!$ as we can change the numbering of the vertices $X\setminus B$. By Robbins estimates similar to Stirling's formula, $p!\ge c p^{p+1/2}e^{-p}>c (p/e)^{p}$, where $c>0$ is a universal constant. As $n\ge N^{\frac 12}$, the number of solutions, $|T|$, satisfies 
    \begin{equation*}
        |T|\ge (n-m)! \ge c((N^{\frac 12}-m)/e)^{N^{\frac 12}-m}.
    \end{equation*}
    The claim then follows from Proposition \ref{prop:Grover-with-unknown-number-of-solutions}.
    \item In a single application of Algorithm~\ref{alg:Grover-with-unknown-number-of-solutions} a solution is found with probability at least $1/2$ by Proposition~\ref{prop:Grover-with-unknown-number-of-solutions}. As any two applications of Algorithm~\ref{alg:Grover-with-unknown-number-of-solutions} are statistically independent, the probability of not finding a solution $M$ times in a row is less than $1/2^M \le \delta$.
\end{enumerate}

The number of qubits and gates in subroutine ORACLE is apparent from its description; see Algorithm~\ref{alg:quantum-oracle}.
\end{proof}

\section[NP-completeness of the restricted inverse travel time\newline problem]{NP-completeness of the restricted inverse travel time problem}

We will now consider the computational complexity of the restricted inverse travel time problem for graphs (Problem~\ref{prob:restricted-inverse-travel-time}). As is customary in computational complexity theory, we will consider a decision version of the problem. The decision problem asks that given an instance of Problem~\ref{prob:restricted-inverse-travel-time} specified by data $(X, V, E_0, E_1, d_{0,V})$, does there exist at least one solution to the problem.

A procedure that solves the original formulation of the restricted inverse travel time problem, i.e., a procedure that finds a solution $E$ if such exists, trivially also solves the decision version. On the other hand, a procedure that solves the decision version can be turned into an algorithm that produces a solution. Namely, suppose that $\mathrm{SOLUTION\_EXISTS}$ is a procedure that with input $(X,V,E_0,E_1,d_{0,V})$ returns either true or false according to whether the instance of Problem~\ref{prob:restricted-inverse-travel-time} specified by this input has at least one solution. Then Algorithm~\ref{alg:sol-from-decision} below produces a solution---provided one exists---by calling $\mathrm{SOLUTION\_EXISTS}$ at most $|E_1|-|E_0|+1$ times.
\vfill

\begin{algorithm}[H]
\caption{Find a solution by solving decision problems}\label{alg:sol-from-decision}
\begin{algorithmic}

\If{$\neg \mathrm{SOLUTION\_EXISTS}(X,V,E_0,E_1,d_{0,V})$}
    \State \Output ``there are no solutions'' and halt
\EndIf

\While{$E_0\neq E_1$}
    \State pick arbitrary $\{x,y\}\in E_1\setminus E_0$
    \State $E_0'\gets E_0\cup\{x,y\}$
    \If{$\mathrm{SOLUTION\_EXISTS}(X,V,E_0',E_1,d_{0,V})$}
        \State $E_0 \gets E_0'$
    \Else
        \State $E_1 \gets E_1 \setminus\{x,y\}$
    \EndIf
\EndWhile
\State \Output $E_0$
\end{algorithmic}
\end{algorithm}

For the purpose of this section, the following three notions from computational complexity theory are relevant: the complexity class P, the complexity class NP, and the class of NP-complete problems. These classes measure computational complexity in terms of classical algorithms (as opposed to quantum algorithms). We informally review these notions below; for formal definitions we refer the reader to~\cite{MR2500087,MR1251285}.

The complexity class P consists of those problems whose solutions can be computed in time that is polynomial in the length of the input to the problem. Here `time' refers to the number of steps required by a Turing machine to produce a solution. Usually this class of problems is equated with the class of problems that are efficiently solvable in practice.

The complexity class NP consists of decision problems for which a ``yes'' answer has a short (polynomial length) ``certificate'' (or a ``proof'') that can be verified in polynomial time. The decision version of Problem~\ref{prob:restricted-inverse-travel-time} is in NP, because the certificate can be taken to be a set $E$ that solves the problem. Verifying that $E$ is a solution to the problem can be done by checking that $E_0\subset E \subset E_1$ and that $d_E |_V = d_{0,V}$, and these verifications can be done in polynomial time. 

A problem in NP is NP-complete if any other problem in NP is reducible (in polynomial time) to the problem. This means that a procedure that solves an NP-complete problem can be used to solve any other problem in NP. In this sense, NP-complete problems are the hardest problems in NP. 
Well-known NP-complete problems include the Boolean satisfiability problem and the decision version of the travelling salesman problem~\cite{garey}.

There is no known polynomial time algorithm for any NP-complete problem, and therefore NP-complete problems are often thought to be intractable in practice~\cite{garey} (although sometimes efficient methods to find good enough approximate solutions exist~\cite{hochba1997approximation}). It is generally thought that no polynomial time algorithm for an NP-complete problem exists, but there is no proof (this is the P vs.\ NP problem, one of the major open problems in theoretical computer science).

We emphasize the difference between computational complexity of a problem and time complexity of an algorithm: Computational complexity is an intrinsic property of a problem; it provides a lower bound for the run time of \emph{any} algorithm that solves the problem. Time complexity of an algorithm, on the other hand, is only a property of that specific algorithm.

Following theorem implies that there is a polynomial time classical algorithm for Problem~\ref{prob:restricted-inverse-travel-time}, if and only if P$=$NP:

\begin{theorem}\label{thm: NP complete}
  The decision version of the restricted inverse travel time problem for graphs (Problem~\ref{prob:restricted-inverse-travel-time}) is NP-complete.
\end{theorem}

We prove Theorem \ref{thm: NP complete} in section \ref{sec: reduction to SAT-3}.

\section{Proofs of main results}
\subsection{Proof of  Theorem \ref{tree-rigidity}} \label{subsection-rigidity}
\smallskip
The proof of Theorem \ref{tree-rigidity} relies on the following key Proposition \ref{tree-monotone}. This proposition can be seen as a combinatorial analogue to the minimal filling problem for manifolds; see e.g. \cite{Gromov,Ivanov01,BI,Ivanov10}.

\begin{proposition} \label{tree-monotone}
Let $T_1=(X_1,E_1,B_1)$ and $T_2=(X_2, E_2, B_2)$ be two trees with $|B_1|=|B_2|$, where $B_1,B_2$ are the sets of all leaves of $T_1,T_2$. 
Denote by $d_1,d_2$ the graph distance function on $T_1,T_2$.
If there is a bijection $\sigma:B_1\to B_2$ such that 
$$d_1(p,q)\geq d_2(\sigma(p),\sigma(q)), \; \textrm{ for all }p,q\in B_1,$$
then $|X_1|\geq |X_2|$.
Moreover, the equality $|X_1|= |X_2|$ holds if and only if $T_1$ is isomorphic to $T_2$ via an isomorphism extending $\sigma$.
\end{proposition}

\begin{proof}
The proof is done by induction on $|B_1|(=|B_2|)$.
The case with $|B_1|=2$ is trivial.
When $|B_1|=3$, there can only be one vertex with degree larger than $2$. Then the total number of edges of the tree is simply half of the sum of all leaf-to-leaf distances. Hence the condition that $d_1\geq d_2$ on leaves implies that $|E_1|\geq |E_2|$, which for trees shows $|X_1|\geq |X_2|$. Furthermore, the equality case $|X_1|=|X_2|$ implies that $d_1=d_2$ on all pair of leaves, which determines a unique tree structure due to the Tree-metric Theorem (Lemma \ref{tree-metric}).

By induction, assume that the proposition is true for $|B_1|\leq m-1$, and consider the case when $|B_1|=m$.
Let us use a simpler notation $\bar p=\sigma(p)$ for $p\in B_1$. Given a leaf $p$ of a tree $T$, we denote by $T(p)$ the maximal (in length) simple path $p x_1\cdots x_N y$ such that $\textrm{deg}_T (x_i) = 2$ for all $i=1,\cdots, N$.
Intuitively, this means that $y$ is the first vertex where the path from the leaf $p$ starts to branch.
(Recall that a path is simple if the vertices in the path do not repeat.)
Denote by $|T(p)|$ the length of the path $T(p)$.
First we prove the following claim:

\smallskip
\textbf{Claim ($\ast$).} There exists a leaf $p\in B_1$ such that $|T(p)|\geq |T(\bar p)|$.

\smallskip
The claim can be proved as follows. 
Consider the following quantity
\begin{equation}
(p,q)_z:=d_1(p,z)+d_1(q,z)-d_1(p,q).
\end{equation}
We remark that this quantity, sometimes called the Gromov product, was considered in e.g. \cite{Buneman74,Evans08}, \cite[Chapter 8.4]{BBI}.

We search for $p,q,z\in B_1$ such that 
\begin{equation}
(p,q)_z \textrm{ is maximal among all choices of }p,q,z\in B_1.
\end{equation}
Then by maximality, one can show that the paths $T(p)$ and $T(q)$ join at the non-leaf end vertex, so that 
\begin{equation} \label{path-join}
d_1(p,q)=|T(p)|+|T(q)|.
\end{equation}
(This fact can be seen as follows. Consider the unique shortest path $[pz]$, meaning from $p$ to $z$, and the shortest path $[qz]$. The two paths must coincide already before reaching $z$, since $z$ has degree $1$ and the simple path between any pair of vertices in a tree is unique. Denote by $v$ the vertex where the two paths first join. Then the quantity $(p,q)_z$ above is equal to $2d(v,z)$. We claim that all vertices on the paths $[pv]$ and $[qv]$ have degree $2$ except $p,q,v$, in which case \eqref{path-join} is satisfied. Otherwise, if there exists a vertex $v'\neq v$ of degree at least $3$ on $[pv]$, then the tree branches out from $[pv]$ at $v'$, and one can take any leaf $b'$ on this branched subtree.
Then 
the quantity $(p,b')_z$ is equal to $2d(z,v')$ which is larger than $2d(z,v)$, contradicting maximality.)

Then one can see that one of the following situations must happen:
$$\textrm{either }|T(p)|\geq |T(\bar p)|, \;\textrm{ or }|T(q)|\geq |T(\bar q)|.$$
Indeed, if none of the above happens, then
$$d_1(p,q)=|T(p)|+|T(q)|<|T(\bar p)|+|T(\bar q)|\leq d_2(\bar p,\bar q),$$
which is a contradiction. The claim ($\ast$) is proved.

\smallskip
Using the claim ($\ast$), we remove the path $T(p), T(\bar p)$ from $T_1,T_2$ respectively. 
The resulting graphs are still trees, with one less leaf. We denote by $T_1-T(p)$ and $T_2-T(\bar p)$ the remaining subtrees. The distance relation $d_1\geq d_2$ on the remaining leaves is still preserved, since the shortest path between the remaining leaves does not go through $T(p)$ and $T(\bar p)$.
Hence by the induction assumption, 
$$\textrm{number of vertices of }T_1-T(p)\geq \textrm{ number of vertices of }T_2-T(\bar p).$$
Now this yields $|X_1|\geq |X_2|$, using the claim ($\ast$) that $|T(p)|\geq |T(\bar p)|$. This proves the inequality part of the proposition.

\medskip
For the equality case, we need a little modification of the argument above. 
If $d_1=d_2$ on all pair of leaves, then $X_1$ is isomorphic to $X_2$ by the Tree-metric Theorem (Lemma \ref{tree-metric}). Otherwise, there exist some $p_0,q_0\in B_1$ such that 
\begin{equation}\label{strict-data}
d_1(p_0,q_0)>d_2(\bar p_0,\bar q_0).
\end{equation}
We need to show that in this case $|X_1|>|X_2|$. The point is to repeat the induction process above until one finds a strict inequality.

Following the argument above, we search for $p,q\in B_1$ such that $d_1(p,q)=|T(p)|+|T(q)|$.
The following three cases may happen.

\smallskip
(1) If $\{p,q\}=\{p_0,q_0\}$, then one of the following situations, similar as above but in the strict sense, must happen:
$$\textrm{either }|T(p_0)|> |T(\bar p_0)|, \;\textrm{ or }|T(q_0)|> |T(\bar q_0)|.$$
Indeed, if none of the above happens, then
$$d_1(p_0,q_0)=|T(p_0)|+|T(q_0)|\leq|T(\bar p_0)|+|T(\bar q_0)|\leq d_2(\bar p_0,\bar q_0),$$
which is a contradiction to \eqref{strict-data}.
Now suppose $|T(p_0)|> |T(\bar p_0)|$ happens.
Then we remove the paths $T(p_0),T(\bar p_0)$ from the graphs, and
this strictly inequality leads to $|X_1|>|X_2|$ after using the induction assumption.

(2) Consider the case when only one of $p_0, q_0$ is in $\{p,q\}$, say, we have picked out $p_0,q$ for some other $q\neq q_0$. If it turns out $|T(p_0)|> |T(\bar p_0)|$, then we simply remove $T(p_0),T(\bar p_0)$ and get a strict inequality.
If not, i.e., $|T(p_0)|\leq |T(\bar p_0)|$, the distance relation $d_1\geq d_2$ on leaves and \eqref{path-join} imply $|T(q)|\geq |T(\bar q)|$. Then we remove $T(q),T(\bar q)$. Although the latter situation does not give a strict inequality, one can still proceed the induction further.

(3) If $p_0,q_0 \notin \{p,q\}$, then proceed the induction further as in the original argument.

\smallskip
One sees that this procedure ends when case (1) or the first situation in case (2) happens, where one gets a strict inequality. This proves $|X_1|>|X_2|$ under \eqref{strict-data}.

\smallskip
Now for the equality case of the proposition, if $|X_1|=|X_2|$, then $d_1=d_2$ on all pair of leaves (the bijection $\sigma$ is fixed); indeed, if there exists a pair $p_0,q_0\in B_1$ such that $d_1(p_0,q_0)>d_2(\bar p_0,\bar q_0)$, the argument above yields $|X_1|>|X_2|$. Then the Tree-metric Theorem (Lemma \ref{tree-metric}) implies that $T_1$ is isomorphic to $T_2$ via an isomorphism extending $\sigma$.
The proof of Proposition \ref{tree-monotone} is concluded.
\end{proof}

Now we apply Proposition \ref{tree-monotone} to prove Theorem \ref{tree-rigidity}.

\begin{proof}[Proof of Theorem \ref{tree-rigidity}]
Let $T_G$ denote a spanning tree of $G$, i.e., a tree subgraph of $G$ containing all vertices of $G$. 
Recall that $T_G$ exists since $G$ is connected.
Denote by $d_{T_G}$ the distance function on $T_G$. Since $T_G$ is a subgraph of $G$, then $d\leq d_{T_G}$. 
This gives 
\begin{equation} \label{eq-monotone}
d_{T_G}|_{B \times B} \geq \bar d|_{\bar B \times \bar B}.
\end{equation}
Vertices in $B$ still have degree $1$ in $T_G$ since they have degree $1$ in $G$.

Note that it may seem that the vertices in $X\setminus B$ could also have degree $1$ in $T_G$.
However, this cannot happen in the class with fixed total number of vertices due to Proposition \ref{tree-monotone}. Indeed, consider the subtree of $T_G$ only having vertices $B$ as leaves, and then \eqref{eq-monotone} and Proposition \ref{tree-monotone} already determine that the subtree has at least $|\bar X|=n$ number of vertices. Thus, any additional leaf in $T_G$ will make the number of vertices in $X$ larger than $n$, which is a contradiction to $|X|=n$.

Now we have shown that $T_G$ is a tree with the set of all leaves $B$. Since $|X|=|\bar X|=n$, then \eqref{eq-monotone} and the equality case of Proposition \ref{tree-monotone} show that $T_G=T$ via an isomorphism extending the identification between $B$ and ${\bar B}$.

At last, we show that $G$ can only be a tree, i.e., $G=T_G$.
This is because adding any edge to a tree (while keeping the graph simple) decreases at least one leaf-to-leaf distance. Indeed, suppose the edge is added to a tree between vertices $v_1,v_2$. Then one can take the shortest path between $v_1,v_2$ in $T_G$, and extend it to a shortest path between a pair of leaves.
Clearly, the length of this path between the leaves strictly decreases when the edge is added. 
Assume that $G$ is not a tree.
Since we have shown that $T_G=T$ via an isomorphism extending the identification between $B,{\bar B}$, then there would exist $b_1,b_2\in B$ such that $d(b_1,b_2)<d_{T_G}(b_1,b_2)=\bar d(\bar b_1,\bar b_2)$, a contradiction.
\end{proof}

The following is the uniqueness part of the classical Tree-metric Theorem 
(e.g. \cite{Buneman71,Buneman74,HY,S69,Z65}) 
which we have used in the proofs.

\begin{lemma}[Uniqueness part of the classical Tree-metric Theorem] \label{tree-metric}
Let $T_1=(X_1,E_1,B_1)$ and $T_2=(X_2, E_2, B_2)$ be two trees with $|B_1|=|B_2|$, where $B_1,B_2$ are the sets of all leaves of $T_1,T_2$. If there is a bijection $\sigma:B_1\to B_2$ such that $d_1|_{B_1\times B_1}=d_2\circ \sigma|_{B_1\times B_1}$, then $T_1$ is isomorphic to $T_2$ via an isomorphism extending $\sigma$.
\end{lemma}

\subsection{Proof of Theorem \ref{thm: NP complete}}
\label{sec: reduction to SAT-3}

We will begin by describing a polynomial time algorithm that takes an instance of a Boolean satisfiability problem (SAT), i.e., a Boolean formula $F$ in conjunctive normal form, to an instance of the restricted inverse travel time problem for graphs, i.e., to the data 
\begin{equation}
    \label{eq:data-for-restricted-travel-time-problem}
    \mathcal{D}=(X,V,E_0,E_1,d_{0,V})
\end{equation} 
in Problem~\ref{prob:restricted-inverse-travel-time}. After that, in Proposition~\ref{prop:reduction} below, we will show that the map $F\mapsto\mathcal{D}(F)$ is in fact a reduction, i.e., formula $F$ is satisfiable, if and only if the instance of Problem~\ref{prob:restricted-inverse-travel-time} specified by $\mathcal{D}(F)$ admits a solution. The NP-completeness of the restricted inverse travel time problem for graphs then readily follows from the reduction.

Recall that a Boolean formula $F$ in conjunctive normal form over variables $u_j$, $j=1,2,\ldots, N$, with $M$ clauses denoted by $C_i$, $i = 1,2,\ldots, M$, is an expression of the form
\begin{subequations}\label{eq:formula}
\begin{equation}
    F = C_1 \land C_2 \land \cdots \land C_M,
\end{equation}
where
\begin{equation}
    C_i = \ell_{1,i} \lor \ell_{2,i} \lor \cdots \ell_{k_i,i},
\end{equation}
\end{subequations}
and for each $i$ and $j$ either $\ell_{j,i} = u_k$ or $\ell_{j,i} = \bar{u}_k$ for some $k$ (here $\bar{u}_k$ is shorthand for the negation $\lnot u_k$ of $u_k$). The Boolean satisfiability problem is to decide whether there exists an assignment of truth values to the variables $u_j$ that makes formula $F$ evaluate to true.

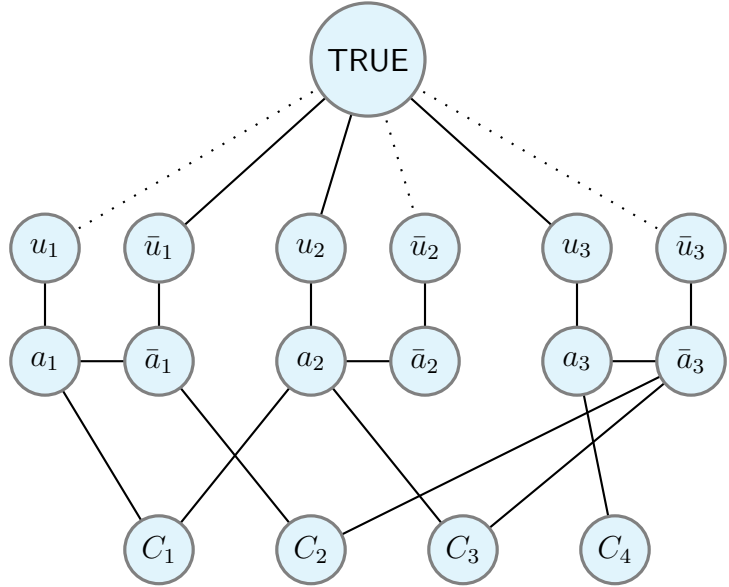
\begin{figure}[h]
    \centering
    \begin{tikzpicture}
\begin{scope}[thick]
\draw[loosely dotted] 
    (4.25,6.50) -- (0.0, 4.0);
    \draw (4.25,6.50) -- (1.5, 4.0);
    \draw (4.25,6.50) -- (3.5, 4.0);
    \draw[loosely dotted] (4.25,6.50) -- (5.0, 4.0);
    \draw (4.25,6.50) -- (7.0, 4.0);
    \draw[loosely dotted] (4.25,6.50) -- (8.5, 4.0);

\draw (0.0,4.0) -- (0.0,2.5);
        \draw (0.0,2.5) -- (1.5,2.5);
        \draw (1.5,2.5) -- (1.5,4.0);

        \draw (3.5,4.0) -- (3.5,2.5);
        \draw (3.5,2.5) -- (5.0,2.5);
        \draw (5.0,2.5) -- (5.0,4.0);

        \draw (7.0,4.0) -- (7.0,2.5);
        \draw (7.0,2.5) -- (8.5,2.5);
        \draw (8.5,2.5) -- (8.5,4.0);
                \draw (0.0,2.5) -- (1.5,0);
        \draw (3.5,2.5) -- (1.5,0);
        \draw (1.5,2.5) -- (3.5,0);
        \draw (8.5,2.5) -- (3.5,0);
        \draw (3.5,2.5) -- (5.5,0);
        \draw (8.5,2.5) -- (5.5,0);
        \draw (7.0,2.5) -- (7.5,0);
    \end{scope}

\draw[draw=gray, line width=1.2pt, fill=cyan!10] (4.25,6.50) circle [radius=0.75] node {$\mathsf{TRUE}$};

    \draw[draw=gray, line width=1.2pt, fill=cyan!10] (0.0,4.0) circle [radius=0.45] node {$u_1$};
    \draw[draw=gray, line width=1.2pt, fill=cyan!10] (1.5,4.0) circle [radius=0.45] node {$\bar{u}_1$};
    \draw[draw=gray, line width=1.2pt, fill=cyan!10] (0.0,2.5) circle [radius=0.45] node {$a_1$};
    \draw[draw=gray, line width=1.2pt, fill=cyan!10] (1.5,2.5) circle [radius=0.45] node {$\bar{a}_1$};

    \draw[draw=gray, line width=1.2pt, fill=cyan!10] (3.5,4.0) circle [radius=0.45] node {$u_2$};
    \draw[draw=gray, line width=1.2pt, fill=cyan!10] (5.0,4.0) circle [radius=0.45] node {$\bar{u}_2$};
    \draw[draw=gray, line width=1.2pt, fill=cyan!10] (3.5,2.5) circle [radius=0.45] node {$a_2$};
    \draw[draw=gray, line width=1.2pt, fill=cyan!10] (5.0,2.5) circle [radius=0.45] node {$\bar{a}_2$};

    \draw[draw=gray, line width=1.2pt, fill=cyan!10] (7.0,4.0) circle [radius=0.45] node {$u_3$};
    \draw[draw=gray, line width=1.2pt, fill=cyan!10] (8.5,4.0) circle [radius=0.45] node {$\bar{u}_3$};
    \draw[draw=gray, line width=1.2pt, fill=cyan!10] (7.0,2.5) circle [radius=0.45] node {$a_3$};
    \draw[draw=gray, line width=1.2pt, fill=cyan!10] (8.5,2.5) circle [radius=0.45] node {$\bar{a}_3$};

    \draw[draw=gray, line width=1.2pt, fill=cyan!10] (1.5,0) circle [radius=0.45] node {$C_1$};
    \draw[draw=gray, line width=1.2pt, fill=cyan!10] (3.5,0) circle [radius=0.45] node {$C_2$};
    \draw[draw=gray, line width=1.2pt, fill=cyan!10] (5.5,0) circle [radius=0.45] node {$C_3$};
    \draw[draw=gray, line width=1.2pt, fill=cyan!10] (7.5,0) circle [radius=0.45] node {$C_4$};
\end{tikzpicture}
    \caption{From $F=(u_1\lor u_2) \land (\bar{u}_1\lor\bar{u}_3) \land (u_2\lor\bar{u}_3) \land u_3$ we construct an instance of the restricted inverse travel time problem (Problem~\ref{prob:restricted-inverse-travel-time}). The set $E_0$ of edges that must be present in the solution is taken to be the empty set; the set $E_1$ of edges that are allowed to appear in the solution consists of the solid and the dotted lines. The distance from $u_j$ to $\bar{u}_j$ and from $\mathsf{TRUE}$ to $C_i$ is required to be $3$ for all $i$ and $j$. A solution for this instance is the subgraph drawn with solid lines; in the corresponding assignment that satisfies $F$ the literals connected with a solid line to $\mathsf{TRUE}$ are true.}
  \label{fig:graph}
\end{figure}

Given a Boolean formula $F$ in conjuctive normal form, we construct an instance $\mathcal{D}=\mathcal{D}(F)$ of Problem~\ref{prob:restricted-inverse-travel-time}, where $\mathcal{D}$ is as in~\eqref{eq:data-for-restricted-travel-time-problem}, as follows (see figure~\ref{fig:graph}): 

The set $X = X(F)$ of vertices consists of following $1 + 4N + M$ elements:
\begin{equation*}
    \begin{split}
    X := \{ \mathsf{TRUE},& u_1, a_1, \bar{a}_1, \bar{u}_1, u_2, a_2, \bar{a}_2, \bar{u}_2, \ldots, u_N, a_N, \bar{a}_N, \bar{u}_N,\\
    &C_1, C_2, \ldots, C_M\}.
  \end{split}
\end{equation*}
Here $\mathsf{TRUE}$, $a_j$, $\bar{a}_j$, and $C_i$ are arbitrary symbols, and $u_j$
and $\bar{u}_j$ are the variables and their negations from formula $F$. 

The set $V=V(F)$ is
\begin{equation*}
     V := \{ (u_j,\bar{u}_j) : j=1,2,\dots,N\} \cup 
     \{(\mathsf{TRUE},C_i) : i=1,2,\dots,M\}.
\end{equation*}

The set $E_0$ of edges that are required to be part of the solution is defined to be empty: $E_0(F) := \emptyset$. The set $E_1=E_1(F)$ of edges that are allowed to be present in the solution is the union of the following four sets~\eqref{eq:E1}--\eqref{eq:E4}:
\begin{subequations}
\begin{equation}
  \label{eq:E1}
  \bigcup_{j=1}^N \big\{ \{\mathsf{TRUE},u_j\}, \{\mathsf{TRUE}, \bar{u}_j\} \big\},
\end{equation}
\begin{equation}
  \label{eq:E2}
  \bigcup_{j=1}^N \big\{ \{u_j,a_j\}, \{a_j,\bar{a}_j\}, \{\bar{a}_j, \bar{u}_j\} \big\},
\end{equation}
and (below $u_j\in C_i$ (resp.~$\bar{u}_j\in C_i$) means that $u_j$ ($\bar{u}_j$) appears as a literal in the $i$:th clause $C_i$ of $F$)
\begin{equation}
  \label{eq:E3}
  \bigcup_{i=1}^M
  \big\{ \{a_j, C_i\} : j \text{ is such that } u_j\in C_i \big\},
\end{equation}
and
\begin{equation}
  \label{eq:E4}
  \bigcup_{i=1}^M
  \big\{ \{\bar{a}_j, C_i\} : j \text{ is such that }  \bar{u}_j\in C_i \big\}.
\end{equation}
\end{subequations}

The distance data on $V$, $d_{0,V}=d_{0,V}(F)$, is defined to be
\begin{equation}
    \label{eq:distance-function-V}
    d_{0,V}(x,y) = 3\text{ for all } (x,y) \in V.
\end{equation}

\begin{proposition}\label{prop:reduction}
    Consider a Boolean formula $F$ in conjunctive normal form, and let $\mathcal{D}(F) = (X,V,E_0,E_1,d_{0,V})$ be the instance of the restricted inverse travel time problem for graphs (Problem~\ref{prob:restricted-inverse-travel-time}) constructed from $F$ as described above. Then there exists a solution to the instance $\mathcal{D}(F)$ of Problem~\ref{prob:restricted-inverse-travel-time}, if and only if formula~$F$ is satisfiable.
\end{proposition}
\begin{proof}
    Below we will identify a path with its vertex sequence, i.e., $P=x_1x_2 \cdots x_{n+1}$ with $x_i\in X$ is a path of length $n$ if there is an edge between $x_i$ and $x_{i+1}$ for all $i=1,2,\dots,n$. Also, we will denote the graph distance function on the graph $(X,E_1)$ by $d_{E_1}$.
    
    Note that $d_{E_1}(\mathsf{TRUE},C_i) = 3$, and that if $P$ is a path in $(X,E_1)$ that connects $\mathsf{TRUE}$ to $C_i$ and if $P$ is of length three, then there exists an index $j$ such that either
    \begin{equation}
        \label{eq:path1}
        u_j \in C_i \text{ and } P = \mathsf{TRUE}\,u_j\,a_j\,C_i,
    \end{equation}
    or
    \begin{equation}
        \label{eq:path2}
        \bar{u}_j \in C_i \text{ and } P = \mathsf{TRUE}\,\bar{u}_j\,\bar{a}_j\,C_i.
    \end{equation}
    
    Suppose first that formula~$F$ is satisfiable. Consider an assignment satisfying the formula, and let $\mathcal{T}$ be the set of true literals in the assignment (i.e., if $u_j$ is true (resp.\ false) in the assignment, then $u_j\in\mathcal{T}$ ($\bar{u}_j\in\mathcal{T}$)). Let $E=E(\mathcal{T})$ consist of the union of the sets
    \eqref{eq:E2}, \eqref{eq:E3}, \eqref{eq:E4}, and the set
    \begin{equation*}
        \big\{ \{\mathsf{TRUE}, \ell\} : \ell\in \mathcal{T} \big\}.
    \end{equation*}
    Then $E\subset E_1$, and the following hold:
    \begin{itemize}
    \item $d_E(u_j,\bar{u}_j) = 3$ for every $j$. Namely, the path $u_ja_j\bar{a}_j\bar{u}_j$ connects $u_j$ to $\bar{u}_j$ in $(X,E)$, and there are no shorter paths connecting those points in $(X,E)$, because in $(X,E)$ the vertex $\mathsf{TRUE}$ is connected to only one of $u_j$ and $\bar{u}_j$.
    \item $d_E(\mathsf{TRUE}, C_i) = 3$ for every $i$. Namely, for every $i$ there is a literal $\ell\in C_i$ such that $\ell\in\mathcal{T}$. Therefore, one of the paths in~\eqref{eq:path1} or~\eqref{eq:path2} is a path in $(X,E)$, also.
    \end{itemize}
    Consequently, $E$ is a solution to the instance $\mathcal{D}(F)$ of Problem~\ref{prob:restricted-inverse-travel-time}.
    
    Suppose then that the instance $\mathcal{D}(F)$ of Problem~\ref{prob:restricted-inverse-travel-time} has a solution $E$. Define an assignment of truth values by setting the variable $u_j$ to be true, if and only if there is no edge between $\mathsf{TRUE}$ and $\bar{u}_j$ in $(X,E)$.

    Fix $i$. Since $d_E(\mathsf{TRUE},C_i)=3$, there exists a path $P$ in $(X,E)$ for which either~\eqref{eq:path1} or~\eqref{eq:path2} holds. Let us consider the former case first. In this case there is no edge between $\bar{u}_j$ and $\mathsf{TRUE}$ (because this would imply $d_E(u_j,\bar{u}_j) = 2 \neq 3 = d_{0,V}(u_j,\bar{u}_j)$), so $u_j$ was assigned to be true in the assignment. Therefore, $C_i$ evaluates to true. In the latter case $u_j$ was assigned to be false in the assignment, and consequently $C_i$ evaluates to true also in this case. It follows that formula $F$ is satisfiable.
\end{proof}

\begin{proof}[Proof of Theorem~\ref{thm: NP complete}]
    The decision version of Problem~\ref{prob:restricted-inverse-travel-time} is in NP, because for a certificate for a `yes' answer (meaning that there exists a solution) we can take any solution $E$ itself. Namely, given an instance $(X,V,E_0,E_1,d_{0,V})$ of Problem~\ref{prob:restricted-inverse-travel-time} and a certificate $E$, we first check that $E_0\subset E \subset E_1$, and then that $d_E(x,y)=d_{0,V}(x,y)$ for all  $(x,y)\in V$. The latter check can be done in polynomial time by algorithms that are analogous to Algorithms~\ref{alg_paths} and~\ref{alg_classicaltest}. If both of these checks pass, we accept, otherwise we reject.

    The decision version of Problem~\ref{prob:restricted-inverse-travel-time} is NP-hard, because the NP-complete problem SAT~\cite{cook} is polynomial-time reducible to the problem. Namely, given a Boolean formula $F$ in conjunctive normal form, we can in polynomial time construct the instance $\mathcal{D}(F)$ of Problem~\ref{prob:restricted-inverse-travel-time} as described above. By Proposition~\ref{prop:reduction} this construction is a reduction.
\end{proof}

\appendix

\section{Numerical examples on inverse travel time problem for graphs in simple cases}

We have implemented Algorithm~\ref{alg:Grover-with-unknown-number-of-solutions} using the open-source toolkit Qiskit~\cite{Qiskit}, which is a software development kit for working with quantum computers. Our code is available on GitHub~\cite{Github_ILLOY2023}. We have also simulated our code with a classical computer in certain simple instances. 

Figure~\ref{fig:numerical-examples} shows the four instances (A--D) of Problem~\ref{prob:inverse-travel-time} that we tested our code with. In the largest instances C and D the quantum circuit consists of $24$ qubits, consequently the (complex) dimension of the state space $(\mathbb{C}^2)^{\otimes 24}$ of the quantum register is $2^{24}=16\,777\,216$. This exponential growth of the dimension of the state space is the limiting factor for the size of instances that can be simulated with a classical computer.

To test the quantum part of the algorithm (Algorithm~\ref{alg:Grover-L-iterations}), we simulated a quantum computer on a classical computer (Apple MacBook Air with M1 processor and 16 GB of RAM). In the simulation, we computed the state of the quantum register immediately prior to the measurement step in Algorithm~\ref{alg:Grover-L-iterations}. Algorithm~\ref{alg:Grover-with-unknown-number-of-solutions} calls Algorithm~\ref{alg:Grover-L-iterations} with various values of parameter $L$ (the number of Grover iterations), of those values, in the simulation we chose to use the one that maximizes the probability of obtaining a correct solution. Computation time for the simulations ranged from less than five seconds (instances A and B) to $118$~minutes (instance~C) to $829$~minutes (instance~D). Figure~\ref{fig:simulation-results} shows the results of the simulations.

\begin{figure}
    \centering
    \begin{subfigure}[b]{0.45\textwidth}
        \centering
        \begin{tikzpicture}
            \begin{scope}[thick]
                \draw (1.5,1.5) -- (0.0, 1.5); 
                \draw (1.5,1.5) -- (3.0, 1.5); 

                \path (0.0,0.0) -- (0.0, 2.6); 
            \end{scope}
            \draw[draw=gray, line width=1.2pt, fill=cyan!20] (0.0,1.5) circle [radius=0.35] node {$2$};

            \draw[draw=gray, line width=1.2pt, fill=cyan!20] (3.0,1.5) circle [radius=0.35] node {$1$};

            \draw[draw=gray, line width=1.2pt, fill=red!15] (1.5,1.5) circle [radius=0.35] node {$3$};
        \end{tikzpicture}
        \caption{}
        \label{fig:instance-A}
    \end{subfigure}
    \hfill
    \begin{subfigure}[b]{0.45\textwidth}
        \centering
        \begin{tikzpicture}
            \begin{scope}[thick]
                \draw (1.0,1.0) -- (2.0, 2.0);
                \draw (1.0,1.0) -- (0.0, 2.0);
                \draw (1.0,1.0) -- (0.0, 0.0);
                \draw (0.0,2.0) -- (0.0, 0.0);
            \end{scope}
            \draw[draw=gray, line width=1.2pt, fill=cyan!20] (2.0,2.0) circle [radius=0.35] node {$1$};

            \draw[draw=gray, line width=1.2pt, fill=cyan!20] (0.0,2.0) circle [radius=0.35] node {$2$};

            \draw[draw=gray, line width=1.2pt, fill=cyan!20] (0.0,0.0) circle [radius=0.35] node {$3$};

            \draw[draw=gray, line width=1.2pt, fill=red!15] (1.0,1.0) circle [radius=0.35] node {$4$};
        \end{tikzpicture}
        \caption{}
        \label{fig:instance-B}
    \end{subfigure}
    \hfill
    \begin{subfigure}[b]{0.45\textwidth}
        \centering
        \begin{tikzpicture}
            \begin{scope}[thick]
                \draw (1.0,1.0) -- (2.0, 2.0);
                \draw (1.0,1.0) -- (0.0, 2.0);
                \draw[dashed] (1.0,1.0) -- (0.0, 0.0);
                \draw[dashed] (1.0,1.0) -- (2.0, 0.0); 

                \draw (0.0,2.0) -- (0.0, 0.0);
                \draw (0.0,0.0) -- (2.0, 0.0);
                \draw (2.0,0.0) -- (2.0, 2.0);
            \end{scope}
            \draw[draw=gray, line width=1.2pt, fill=cyan!20] (2.0,2.0) circle [radius=0.35] node {$1$};

            \draw[draw=gray, line width=1.2pt, fill=cyan!20] (0.0,2.0) circle [radius=0.35] node {$2$};

            \draw[draw=gray, line width=1.2pt, fill=cyan!20] (0.0,0.0) circle [radius=0.35] node {$3$};

            \draw[draw=gray, line width=1.2pt, fill=cyan!20] (2.0,0.0) circle [radius=0.35] node {$4$};

            \draw[draw=gray, line width=1.2pt, fill=red!15] (1.0,1.0) circle [radius=0.35] node {$5$};
        \end{tikzpicture}
        \caption{}
        \label{fig:instance-C}
    \end{subfigure}
    \hfill
    \begin{subfigure}[b]{0.45\textwidth}
        \centering
        \begin{tikzpicture}
            \begin{scope}[thick]
                \draw (1.0,1.0) -- (2.0, 2.0);
                \draw (1.0,1.0) -- (0.0, 2.0);
                \draw (1.0,1.0) -- (0.0, 0.0);
                \draw (1.0,1.0) -- (2.0, 0.0);
            \end{scope}
        
\draw[draw=gray, line width=1.2pt, fill=cyan!20] (2.0,2.0) circle [radius=0.35] node {$1$};

            \draw[draw=gray, line width=1.2pt, fill=cyan!20] (0.0,2.0) circle [radius=0.35] node {$2$};

            \draw[draw=gray, line width=1.2pt, fill=cyan!20] (0.0,0.0) circle [radius=0.35] node {$3$};

            \draw[draw=gray, line width=1.2pt, fill=cyan!20] (2.0,0.0) circle [radius=0.35] node {$4$};

            \draw[draw=gray, line width=1.2pt, fill=red!15] (1.0,1.0) circle [radius=0.35] node {$5$};
        \end{tikzpicture}
        \caption{}
        \label{fig:instance-D}
    \end{subfigure}
    \hfill
    \begin{subfigure}[t]{0.65\textwidth}
        \centering
        \begin{tabular}{ c c | c | c | c | c}
            $b_1$ & $b_2$ & $d_A(b_1,b_2)$ & $d_B(b_1,b_2)$ & $d_C(b_1,b_2)$ & $d_D(b_1,b_2)$\\
            \hline
            $1$&  $2$& $2$& $2$& $2$& $2$\\
            $1$&  $3$& $-$& $2$& $2$& $2$\\
            $1$&  $4$& $-$& $-$& $1$& $2$\\
            
            $2$&  $3$& $-$& $1$& $1$& $2$\\
            $2$&  $4$& $-$& $-$& $2$& $2$\\

            $3$&  $4$& $-$& $-$& $1$& $2$\\
        \end{tabular}
        \caption{Boundary distance data for instances~A--D. The two leftmost columns contain the labels of vertices ${b_1,b_2}\in B$, the other columns contain the distances $d(b_1,b_2)$ for the four different instances.}
        \label{table:distance-data}
    \end{subfigure}
       \caption{Figures~{(A)--(D)} depict the four instances (labeled \mbox{A--D}, respectively) of Problem~\ref{prob:inverse-travel-time} that we consider. The boundary vertices $b_j\in B$ are drawn in blue; the boundary distances are tabulated in {(E)}. The dashed edges in {(C)} are edges that do not affect the boundary distances, thus this instance has a nonunique solution.}
       \label{fig:numerical-examples}
\end{figure}
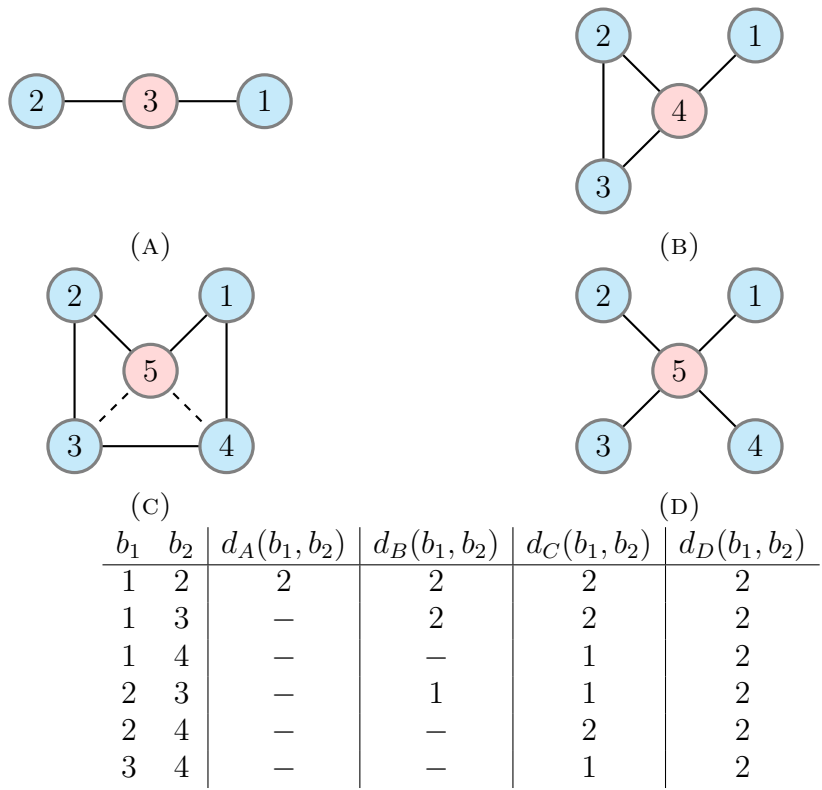

\begin{figure}
    \begin{adjustwidth}{-1.9cm}{-1.9cm}
    \centering
    \begin{subfigure}[b]{0.85\textwidth}
        \centering
        \includegraphics[width=\textwidth]{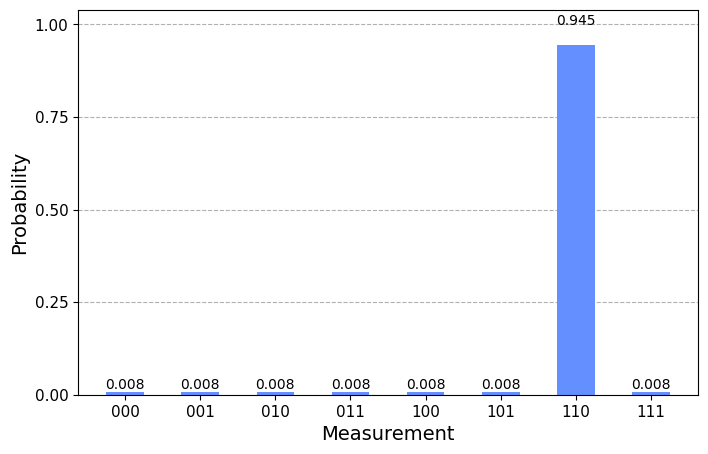}
        \caption{}
        \label{A}
    \end{subfigure}
    \hfill
    \begin{subfigure}[b]{0.335\textwidth}
        \centering
        \includegraphics[width=\textwidth]{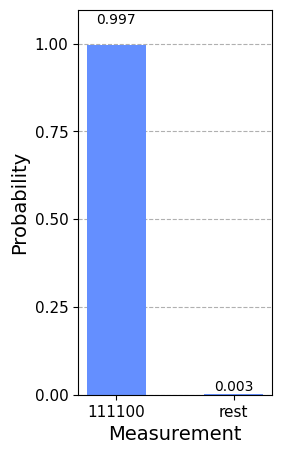}
        \caption{}
        \label{B}
    \end{subfigure}
   
    \begin{subfigure}[b]{0.85\textwidth}
        \centering
        \includegraphics[width=\textwidth]{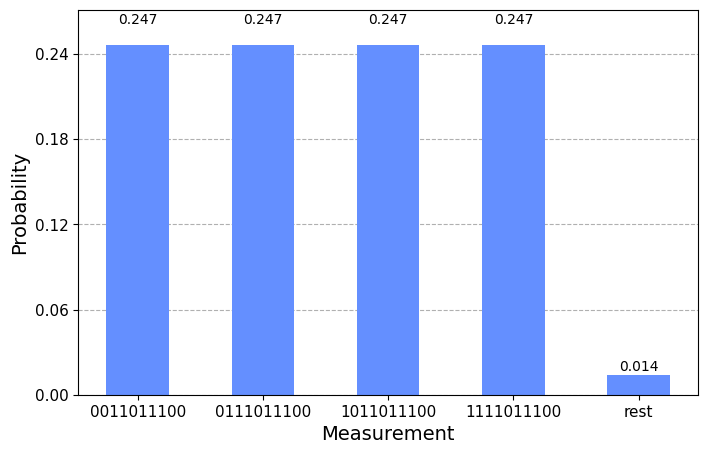}
        \caption{}
        \label{C}
    \end{subfigure}
    \hfill
    \begin{subfigure}[b]{0.335\textwidth}
        \centering
        \includegraphics[width=\textwidth]{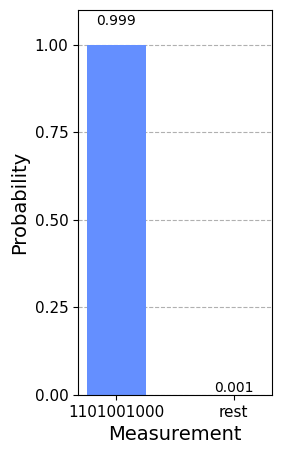}
        \caption{}
        \label{D}
    \end{subfigure}
    
    \vspace{0.2cm}

    \begin{subfigure}[b]{1\textwidth}
        \centering
        \begin{tabular}{ c | c | c | c | c }
            instance & $L$ & $n$ qubits & dim.\ of state space \\
            \hline
            A &   $2$ &  $8$ & $256$\\
            B &  $6$ &  $15$ & $32\,768$\\
            C &  $13$ & $24$ & $16\,777\,216$\\
            D &  $25$ & $24$ & $16\,777\,216$\\
        \end{tabular}
        \caption{}
    \end{subfigure}
       \caption{Plots~{(A)--(D)} display simulation results for instances {A--D}, respectively. In {(A)} the probability of measurement for each value of $\vec{e}$ (the adjacency matrix of a $3$-vertex graph) is shown. In {(B)--(D)} the probabilities of those values of $\vec{e}$ that correspond to a solution of the instance are shown; ''rest`` is the sum of the probabilities of those values that are not a solution. The bitstring denoting the measurement result consists of the elements of $e(j,k)$, $1\le j<k\le n$, of $\vec{e}$  in reverse lexicographical order so that, e.g., in the case $n=3$ the string corresponds to $e(2,3)\,e(1,3)\,e(1,2)$. Table~{(E)} shows the value of parameter $L$ (number of Grover iterations), the number of qubits (including ancillae), and the dimension of the state space of the qubit register, for each of the instances~{A--D}.}
       \label{fig:simulation-results}
    \end{adjustwidth}
\end{figure}

\clearpage

\section{Grover's algorithm}
\label{app: Grover}

We give a review of Grover's algorithm with main focus on the case where the number of solutions is unknown. The approach here follows the outline explained in~\cite{Aaronson}, for other approaches, see~\cite{boyer1998tight} or~\cite[Section~6.3]{MR1796805}. Grover's original analysis~\cite{MR1427516} is for the case of a unique solution.

As in Section~\ref{sec:grover}, let $N$ be a positive integer, $T\subset\{0,1\}^N$ a subset of solutions and $F=\{0,1\}^N \setminus T$ its complement, and $O_T$ a quantum oracle for the set $T$. Assume that $O_T$ uses $k$ ancillae and recall that the action of $O_T$ is given by~\eqref{eq:oracle-operator}. 

Note that if the target qubit $\kkket{r}$ in~\eqref{eq:oracle-operator} is prepared at the superposition $(\kkket{0}-\kkket{1}) / \sqrt{2}$ instead of $\kkket{0}$ or $\kkket{1}$, then $O_T$ acts as
\begin{equation*}
    O_T \left(\kkket{x}\otimes\frac{\kkket{0}-\kkket{1}}{\sqrt{2}}\otimes\kkket{0}_k\right)
    = \Big(\widetilde{O}_T\kkket{x}\Big) \otimes\frac{\kkket{0}-\kkket{1}}{\sqrt{2}}\otimes\kkket{0}_k,
\end{equation*}
where the operator $\widetilde{O}_T$ acts on a computational basis state $\kkket{x}\in(\C^2)^{\otimes N}$ as
\begin{equation*}
    \widetilde{O}_T \kkket{x}=
    \begin{cases}
        -\kkket{x}, & x \in T,\\
        \phantom{-}\kkket{x}, & x \in F.
    \end{cases}
\end{equation*}
The operator $\widetilde{O}_T$ is called a \emph{phase oracle}.

If necessary, by increasing the number $k$ of ancillae for the oracle $O_T$, the diffusion subroutine can be assumed to operate with the same $k$ ancillae as $O_T$. The diffusion subroutine acts on a computational basis state of the form $\kkket{x}\otimes\kkket{r}\otimes\kkket{0}_k$ as
\begin{equation*}
    \kkket{x} \otimes \kkket{r} \otimes \kkket{0}_k \mapsto 
    \Big(D_N\kkket{x}\Big) \otimes \kkket{r} \otimes \kkket{0}_k,
\end{equation*}
where
\begin{equation*}
  D_N = 2\kkket{\gamma_N}\bbbra{\gamma_N} - \mathrm{id}^{\otimes N},
\end{equation*}
and $\gamma_N$ is as in~\eqref{eq:gamma-N}.

The composition of the oracle and the diffusion subroutine is called \emph{Grover iteration} and it is denoted by $G_N$; it acts as
\begin{equation}
  \label{eq:grover-iteration-operator}
  G_N\left( 
  \kkket{\psi} \otimes \frac{\kkket{0}-\kkket{1}}{\sqrt{2}} \otimes \kkket{0}_k \right)
  = \Big(D_N \widetilde{O}_T\,\kkket{\psi}\Big) \otimes \frac{\kkket{0}-\kkket{1}}{\sqrt{2}} \otimes \kkket{0}_k,
\end{equation}
where $\kkket{\psi} \in (\C^2)^{\otimes N}$ is arbitrary.

Below is a description of Grover's algorithm. An essential question is how many times the Grover iteration should be applied in the algorithm, and this in turn depends on whether the number of solutions $|T|$ is \emph{a priori} known. In Algorithm~\ref{alg:Grover-L-appendix} below we choose to take the number $L$ of Grover iterations to be applied as an input to the algorithm; in Proposition~\ref{prop:Grover-with-known-number-of-solutions} we will show how to choose $L$ if $|T|$ is known, and in Algorithm~\ref{alg:Grover-unknown-appendix} and Proposition~\ref{prop:Grover-with-unknown-number-of-solutions} we will address the case where $|T|$ is unknown. In Algorithm~\ref{alg:Grover-L-appendix} below the $\kkket{\xi}$'s in the comments describe the state of the quantum register at various points during the execution of the algorithm, they will be needed in Proposition~\ref{prop:Grover-evolution-of-state}.

\begin{algorithm}[H]
    \caption{Grover's algorithm with $L$ iterations}
    \label{alg:Grover-L-appendix}
    \begin{algorithmic}
    \Require \Register{Parameter}{$L\in\mathbb{N}$}
    \State initialize all qubits to $|0\rangle$
    \LComment{$\kkket{\xi^{(\mathrm{init})}} := \kkket{0}_N\otimes\kkket{0}\otimes\kkket{0}_k$}
    \State apply the NOT gate $X$ to the $(N+1)$:th qubit
    \State apply the Hadamard gate $H$ to the first $N+1$ qubits
    \LComment{
    $\kkket{\xi^{(0)}} = \big( H^{\otimes N} \otimes HX \otimes \mathrm{id}^{\otimes k} \big)\,\kkket{\xi^{(\mathrm{init})}}$
    }
    \For{$\ell=1,2,\dots,L$}
        \State apply the Grover iteration to the register
        \LComment{
        $\kkket{\xi^{(\ell)}} = G_N \,\kkket{\xi^{(\ell-1)}},\,
        \ell=1,\dots,L$
        }
    \EndFor
    \State \Output measurement of the $N$ first qubits
\end{algorithmic}
\end{algorithm}

Following proposition describes the evolution of the state of the
quantum register during Grover's algorithm:
\begin{proposition}
  \label{prop:Grover-evolution-of-state}
  Assume that the number $|T|$ of solutions satisfies $0<|T|<2^N$. Define
  \begin{equation*}
    \kkket{\alpha} :=
    \frac{1}{\sqrt{|F|}} \sum_{x\in F} \kkket{x}
    \text{ and }
    \kkket{\beta} :=
    \frac{1}{\sqrt{|T|}} \sum_{x\in T} \kkket{x},
  \end{equation*}
  and let $\theta = \theta(|T|,N) \in (0, \pi/2)$ be the number
  defined by
  \begin{equation}
    \label{eq:Grover-theta}
    \sin\theta = \frac{\sqrt{|T|}}{\sqrt{2^N}}.
  \end{equation}
  Let $(\kkket{\xi^{(\ell)}})_{\ell=0}^L$ be the sequence of states of the quantum register during the execution of Algorithm~\ref{alg:Grover-L-appendix}.
  \begin{enumerate}
  \item For $\ell\ge 0$ it holds that
    \begin{equation}
      \label{eq:xi-r}
      \kkket{\xi^{(\ell)}}
      = \Big( \cos\big((2\ell+1)\,\theta\big) \,\kkket{\alpha}
      +\sin\big((2\ell+1)\,\theta\big) \,\kkket{\beta}\Big)
      \otimes \frac{\kkket{0}-\kkket{1}}{\sqrt{2}}
      \otimes\kkket{0}_k.
    \end{equation}
  \item The probability that Grover's algorithm with $L$ iterations
    outputs a correct solution is $\sin^2((2L+1)\theta)$. 
\end{enumerate}
\end{proposition}

\begin{proof}
  A calculation shows that
  \begin{equation*}
    \begin{split}
      \kkket{\xi^{(0)}}
      &= \big(H^{\otimes N} \otimes HX \otimes \mathrm{id}^{\otimes k})\kkket{0}_{N+1+k}\\
      &= \underbrace{\left(\frac{ \kkket{0} + \kkket{1}}{\sqrt{2}}\right)
        \otimes \cdots \otimes \left(\frac{ \kkket{0} + \kkket{1}}{\sqrt{2}}\right)}_{\textrm{$N$ times}} 
        \otimes \frac{\kkket{0}-\kkket{1}}{\sqrt{2}}
        \otimes \kkket{0}_k\\
      &= \kkket{\gamma_N} 
        \otimes \frac{\kkket{0}-\kkket{1}}{\sqrt{2}}
        \otimes \kkket{0}_k,
    \end{split}
  \end{equation*}
  and by~\eqref{eq:grover-iteration-operator} for $\ell=1,2,\ldots,L$,
  \begin{equation*}
    \begin{split}
    \kkket{\xi^{(\ell)}} 
    &= (G_N)^\ell \left(\kkket{\gamma_N}\otimes\frac{\kkket{0}-\kkket{1}}{\sqrt{2}}\otimes\kkket{0}_k\right)\\
    &= \Big((D_N\widetilde{O}_T)^\ell\kkket{\gamma_N}\Big) \otimes \frac{\kkket{0}-\kkket{1}}{\sqrt{2}} \otimes \kkket{0}_k.
    \end{split}
  \end{equation*}
  Therefore, in order to prove~\eqref{eq:xi-r}, it suffices to prove
  that
  \begin{equation}
    \label{eq:gamma-r}
    (D_N\widetilde{O}_T)^\ell\kkket{\gamma_N} = \cos\big((2\ell+1)\,\theta\big)\,\kkket{\alpha}
    +\sin\big((2\ell+1)\,\theta\big) \,\kkket{\beta}
\end{equation}
  for all $\ell\geq0$.

  Since $|F|+|T| = 2^N$ and $0<\theta<\pi/2$, it follows that
  \begin{equation*}
    \cos\theta = \sqrt{1-\sin^2\theta} = \frac{\sqrt{|F|}}{\sqrt{2^N}}.
  \end{equation*}
  Therefore,
  \begin{equation}
    \label{eq:gamma-alpha-beta}
    \kkket{\gamma_N}
    = \frac{1}{\sqrt{2^N}}\sum_{x\in F} \kkket{x} + \frac{1}{\sqrt{2^N}}\sum_{x\in T} \kkket{x}
    = \cos\theta\, \kkket{\alpha} + \sin\theta\, \kkket{\beta},
  \end{equation}
  and~\eqref{eq:gamma-r} holds if $\ell=0$.

  Let us consider the case $\ell\ge 1$. Observe that the plane spanned by
  the orthonormal vectors $\kkket{\alpha}$ and $\kkket{\beta}$ is an
  invariant subspace under the mapping $\widetilde{O}_T$ (since
  $\widetilde{O}_T\kkket{\alpha}=\kkket{\alpha}$ and
  $\widetilde{O}_T\kkket{\beta} = -\kkket{\beta}$), and it is also an invariant
  subspace under the mapping $\kkket{\gamma_N}\bbbra{\gamma_N}$
  (because of~\eqref{eq:gamma-alpha-beta}). Therefore, it is an
  invariant subspace under the mapping $D_N\widetilde{O}_T$, too.  In fact, as we
  will next prove, with respect to the basis
  $(\kkket{\alpha},\kkket{\beta})$ the operator $D_N\widetilde{O}_T$ acts as a
  rotation by $2\theta$:
  \begin{equation*}
    D_N \widetilde{O}_T \big(a\kkket{\alpha} + b\kkket{\beta}\big)
    = a' \kkket{\alpha} + b'\kkket{\beta},
  \end{equation*}
  where the coefficients ${a,b}\in\C$ and ${a',b'}\in\C$ are related
  by
  \begin{equation}
    \label{eq:rotation-matrix}
    \begin{bmatrix}
      a' \\ b'
    \end{bmatrix}
    =
    \begin{bmatrix}
      \cos(2\theta)    & -\sin(2\theta)\\
      \sin(2\theta)    & \phantom{-}\cos(2\theta)
    \end{bmatrix}
    \begin{bmatrix}
      a \\ b
    \end{bmatrix}.
  \end{equation}

  Note that $\bbbraket{\gamma_N}{\alpha} =
  \cos\theta$. Consequently, with~\eqref{eq:gamma-alpha-beta} and the
  double angle formulas for the trigonometric functions we obtain
  \begin{equation*}
    \begin{split}
      D_N \widetilde{O}_T\kkket{\alpha} 
      &= 2\cos\theta\,\kkket{\gamma_N} - \kkket{\alpha}\\
      &= (2\cos^2\theta - 1)\,\kkket{\alpha} + 2\cos\theta\sin\theta\,\kkket{\beta}\\
      &= \cos(2\theta)\,\kkket{\alpha} + \sin(2\theta)\,\kkket{\beta}.
    \end{split}
  \end{equation*}
  This determines the first column of the matrix
  in~\eqref{eq:rotation-matrix}, and an analogous calculation for
  $D_N \widetilde{O}_T\kkket{\beta}$ determines the second column.

  Now~\eqref{eq:gamma-r} for $\ell=1,2,\ldots$ can be proved by induction
  on $\ell$ using~\eqref{eq:rotation-matrix} and the angle sum formulas
  for the trigonometric functions.

    Finally, for $x\in T \subset\{0,1\}^N$, $y\in\{0,1\}$ and $z\in\{0,1\}^k$ we have
  \begin{equation*}
    \bbbraket{xyz}{\xi^{(L)}}
    = 
    \begin{cases}
      0, \text{ if } z \neq 0_k,\\
      \frac{(-1)^y}{\sqrt{2}}\sin((2L+1)\theta) / \sqrt{|T|}, \text{ if } z = 0_k,
    \end{cases}
  \end{equation*}
  where $0_k\in\{0,1\}^k$ is the zero vector and $xyz\in\{0,1\}^{N+1+k}$ is the concatenation of $x$, $y$ and $z$. Therefore, the probability that the result $x$ of the measurement of the first $N$ qubits belongs to $T$ is 
    \begin{equation*}
        \sum_{(x,y,z) \in T \times \{0,1\}^{k+1}} \big|\bbbraket{xyz}{\xi^{(L)}}\big|^2
        = |T| \, \frac{\sin^2((2L+1)\theta)}{|T|}= \sin^2((2L+1)\theta).\qedhere
    \end{equation*}
\end{proof}

Following well-known proposition shows that if $|T|$ is known, then with an appropriate choice of $L$ Algorithm~\ref{alg:Grover-L-appendix} finds a solution with probability at least $1/2$.

\begin{proposition}
  \label{prop:Grover-with-known-number-of-solutions}
  Suppose that the number of solutions $|T|$ is known and that
  $0<|T|<2^N$. Let $\theta$ be defined by~\eqref{eq:Grover-theta} and $L = L(\theta)$ by~\eqref{eq:Grover-L-theta}.

  With probability at least $1/2$ Grover's algorithm with $L$
  iterations (Algorithm~\ref{alg:Grover-L-appendix}) outputs a solution $x\in T$. Furthermore, the algorithm
  queries the oracle at most $\sqrt{2^N / |T|}$ times.
\end{proposition}
\begin{proof}
  Grover's algorithm with $L$ iterations produces a correct solution
  with probability $\sin^2((2L+1)\theta)$ by
  Proposition~\ref{prop:Grover-evolution-of-state}, and the
  choice~\eqref{eq:Grover-L-theta} implies that
  \begin{equation*}
    \frac{\pi}{4} \le (2L+1)\theta < \frac{3\pi}{4}.
  \end{equation*}
  This implies $\sin^2((2L+1)\theta) \ge 1/2$.

  The number of queries the algorithm makes is $L$, which implies the
  desired bound.
\end{proof}

The search problem is most interesting when $|T|\ll 2^N$. In this case
the angle $\theta$ defined by~\eqref{eq:Grover-theta} is small and the
choice~\eqref{eq:Grover-L-theta} implies that
$(2L+1)\theta \approx \pi / 2$. Therefore, if $0<|T|\ll 2^N$ and $L$
is as in~\eqref{eq:Grover-L-theta}, then Grover's algorithm with $L$
iterations finds a solution with probability close to one.

Following variation of Grover's algorithm applies to the case where
the number of solutions to the search problem is unknown. The
algorithm consists of a repeated application of Grover's algorithm
with $L$ iterations as a subroutine, where the value of $L$ is
varied. After each subroutine the obtained value $x$ is tested for
membership in $T$, and if $x\in T$, then $x$ is returned and the
algorithm is terminated. The test for membership is thought of as a
query to a classical oracle, and the query complexity of the algorithm
is the total number of queries to the quantum and classical oracles.

\begin{algorithm}[H]
    \caption{Grover's algorithm for the search problem with an unknown number of solutions}
    \label{alg:Grover-unknown-appendix}
    \begin{algorithmic}
    \For{$L=0,1,2$}
        \State apply Algorithm~\ref{alg:Grover-L-appendix} with parameter $L$ to obtain $x\in\{0,1\}^N$
        \If{$x\in T$}
            \State \Output $x$ and halt
        \EndIf
    \EndFor
    \For{$K=2^{N-4},2^{N-5},\dots,2,1$}
        \State find $\theta_K\in(0,\pi/2)$ such that $\sin\theta_K = \sqrt{K / 2^N}$ 
        \State $L\gets$ \Call{$L$}{$\theta_K$} \Comment{$L(\theta)$ is defined in~\eqref{eq:Grover-L-theta}}
        \State apply Algorithm~\ref{alg:Grover-L-appendix} with parameter $L$ to obtain $x\in\{0,1\}^N$
        \If{$x\in T$}
            \State \Output $x$ and halt
        \EndIf
    \EndFor
    \State \Output ``no solution found''
\end{algorithmic}
\end{algorithm}

Following proposition shows that if there are
solutions, then Algorithm~\ref{alg:Grover-unknown-appendix} finds a solution with probability at least $1/2$:
\begin{proposition}
  \label{prop:Grover-with-unknown-number-of-solutions}
  Consider Grover's algorithm for the search problem with an unknown number of solutions (Algorithm~\ref{alg:Grover-unknown-appendix}).
  \begin{enumerate}
  \item\label{prop:Grover-unknown-item-1} The algorithm queries the
    quantum and classical oracles in total $O(2^{N/2})$ times.
  \item\label{prop:Grover-unknown-item-2} If $|T|>0$, then with
    probability at least $1/2$ the algorithm outputs a solution
    $x \in T$ after querying the oracles at most $C\sqrt{2^N / |T|}$
    times (where $C>0$ is a universal constant).
  \end{enumerate}
\end{proposition}
\begin{proof}
    The first for-loop in Algorithm~\ref{alg:Grover-unknown-appendix} involves in total six queries to the classical or quantum oracles, and in the second for-loop for each $K$ the oracles are queried $L(\theta_K)+1 \le 2^{N/2} / \sqrt{K} + 1$ times. Therefore, the total number of queries is $O(2^{N/2})$, and~\eqref{prop:Grover-unknown-item-1} of the proposition is proved.

    To prove~\eqref{prop:Grover-unknown-item-2}, we may assume that $0<|T|<2^N$.  Let $\theta\in(0,\pi/2)$ be such that $\sin\theta = \sqrt{|T| / 2^N}$. Recall that by Proposition~\ref{prop:Grover-evolution-of-state} Algorithm~\ref{alg:Grover-L-appendix} with input $L$ yields a solution with probability $\sin^2 ((2 L+1)\theta)$.
    
    Let $\ell_0$ be the integer for which $2^{\ell_0-1}\le |T| < 2^{\ell_0}$. A calculation shows that if $\ell_0 \ge N-3$, then either $\theta$, $3\theta$, or $5\theta$ is contained in the interval $[\pi/4,3\pi/4]$. Consequently, if $\ell_0 \ge N-3$, then the algorithm outputs a solution and halts with probability at least $1/2$ during the execution of the first for-loop.
  
    Assume that $\ell_0 \le N-4$ and consider $K=2^{\ell_0}$. Then $0<\theta_K < \pi / 8$, and
    \begin{equation*}
        \frac{\sin\theta}{\sin\theta_K}
        = \frac{\sqrt{|T|}}{\sqrt{2^{\ell_0}}} \in \bigg[\frac{1}{\sqrt{2}}, 1\bigg),
    \end{equation*}
    so we can estimate (using $\theta_K > \theta$, $\sin\theta \le \theta$, and $\sin\theta_K \ge \theta_K / \sqrt{2}$)
    \begin{equation*}
        \frac{1}{2} \le \frac{\theta}{\theta_K} < 1.
    \end{equation*}
    Choice~\eqref{eq:Grover-L-theta} guarantees that with $L_K=L(\theta_K)$ it holds that $\pi/2 \le (2L_K+1)\theta_K < 3\pi/4$, so
    \begin{equation*}
        (2L_K+1)\theta
        = (2L_K+1)\theta_K \cdot\frac{\theta}{\theta_K} 
        \in \left[\frac{\pi}{4}, \frac{3\pi}{4}\right].
    \end{equation*}
    Therefore, with probability at least $1/2$ the algorithm outputs a
    solution and halts at the latest when the second for-loop is executed with $K = 2^{\ell_0}$.

    The total number of queries executed up to and including the loop
    corresponding to $K=2^{\ell_0}$ is bounded from the above by
    \begin{equation*}
        6 + \sum_{j=4}^{N-{\ell_0}} \left( \sqrt{ \frac{2^N}{2^{N-j}} } + 1\right)
        \le C \sqrt{\frac{2^N}{2^{\ell_0}}}
        \le C \sqrt{\frac{2^N}{|T|}}.\qedhere
        \end{equation*}
        \end{proof}
\vfill

\section[The OR gate and the iterated AND and OR \newline operators]{The OR gate and the iterated AND and OR operators}
\label{appendix:iterated-and-or}

\begin{algorithm}[H]

\caption{OR operator}\label{gate_or}
\begin{algorithmic}
\Require 
\Register{Controls}{$C_1$ and $C_2$}\\
\Register{Target}{$T$}

\Gate{NOT}{$C_1$}
\Gate{NOT}{$C_2$}
\Gate{NOT}{$T$}
\Gate{CCNOT}{
\Register{Controls}{$C_1$ and $C_2$}\\
\Register{Target}{$T$}
}
\Gate{NOT}{$C_1$}
\Gate{NOT}{$C_2$}
\end{algorithmic}
\end{algorithm}

\begin{algorithm}[H]
\caption{AND${}_{(m)}$ operator}\label{gate_and-n} 
\begin{algorithmic}
\Require 
\Register{Parameter}{$m\ge 3$}\Comment{``Number of input variables''}\\
{\bf Operates on qubits}\\
\Register{Controls}{$C_1,C_2,\dots,C_m$}\\
\Register{Target}{$T$}
\Comment{``Output''}\\
\Register{Ancillae}{$F(k)$, $k = 1,\dots,m-2$}\\

            \Gate{CCNOT}{
\Register{Controls}{$C_1$ and $C_2$}\\
\Register{Target}{$F(1)$}
}
\vspace{2mm}
  \If{$m\ge 4$}        
 \For{$j = 2,\dots,m - 2$}
               \Gate{CCNOT}{
\Register{Controls}{$F(j-1)$ and $C_{j+1}$}\\
\Register{Target}{$F(j)$}
}
        \EndFor 
 \EndIf
 \vspace{2mm}
                    \Gate{CCNOT}{
\Register{Controls}{$F(m-2)$ and $C_m$}\\
\Register{Target}{$T$}
} \vspace{2mm}
       \If{$m\ge 4$} \Comment{``Next we uncompute ancillae''}
         \For{$j = m-2,\dots,2$}
               \Gate{REVERSE-CCNOT}{
\Register{Controls}{$F(j-1)$ and $C_{j+1}$}\\
\Register{Target}{$F(j)$}
}
        \EndFor 
 \EndIf
 \vspace{2mm}
             \Gate{REVERSE-CCNOT}{
\Register{Controls}{$C_1$ and $C_2$}\\
\Register{Target}{$F(1)$}
}
        \end{algorithmic}
\end{algorithm}

\begin{algorithm}[H]
\caption{OR${}_{(m)}$ operator}\label{gate_or-n} 
\begin{algorithmic}
\Require 
\Register{Parameter}{$m\ge 3$}\Comment{``Number of input variables''}\\
{\bf Operates on qubits}\\
\Register{Controls}{$C_1,C_2,\dots,C_m$}\\
\Register{Target}{$T$}
\Comment{``Output''}\\
\Register{Ancillae}{$F(k)$, $k = 1,\dots,m-2$}\\
 \vspace{2mm}
 \For{$j = 1,\dots,m $}
   \Gate{NOT}{$C_j$}  
    \EndFor 
     \vspace{2mm}
      \Gate{AND${}_{(m)}$}{
\Register{Controls}{ $C_1,\ C_2,\dots,C_m$}\\
\Register{Target}{$T$}\\
\Register{Ancillae}{$F(k)$, $k = 1,\dots,m-2$}
}
 \vspace{2mm}
 \Gate{NOT}{$T$}
 \vspace{2mm}
 \For{$j = 1,\dots,m $}
   \Gate{NOT}{$C_j$}
    \EndFor 
        \end{algorithmic}
\end{algorithm}

\printbibliography

\end{document}